\documentclass[a4paper,12pt]{amsart}

\usepackage{graphicx}
\usepackage[colorlinks, linkcolor= blue, citecolor= red]{hyperref}
\usepackage{pdfsync}
\usepackage{color}

\usepackage[a4paper, hmargin= 3cm, vmargin=2cm]{geometry}

\usepackage[T1]{fontenc}
\usepackage[utf8]{inputenc}

\usepackage[all]{xy}
\usepackage{amsmath, amssymb, amsfonts, amscd, amsthm, mathrsfs}

\usepackage{lmodern}
\usepackage{hyperref}
\numberwithin{equation}{section}

\theoremstyle{plain}
\newtheorem{thm}{Theorem}[section]
\newtheorem{lem}[thm]{Lemma}
\newtheorem{coro}[thm]{Corollary}
\newtheorem{prop}[thm]{Proposition}

\theoremstyle{definition}

\theoremstyle{remark}
\newtheorem{rmk}[thm]{Remark}

\newcommand{\Hom}{{\rm Hom}}

\newcommand{\Gal}{{\rm Gal}}

\newcommand{\sC}{{\mathcal C}}

\newcommand{\sZ}{{\mathcal Z}}
\renewcommand{\H}{{\mathbb H}}

\newcommand{\Q}{{\mathbb Q}}

\newcommand{\U}{{\mathbb U}}

\newcommand{\Z}{{\mathbb Z}}

\newcommand{\Ind}{\text{\rm Ind}}

\newcommand{\f}{\mathbb{F}}
\newcommand{\h}{\operatorname{H}}

\newcommand{\cupp}{\!\smile\!}
\renewcommand{\tilde}{\widetilde}
\renewcommand{\bar}[1]{\overline{#1}{}}
\newcommand{\dz}{\zeta_p}
\newcommand{\sq}[1]{\sqrt[p]{#1}}
\newcommand{\N}{{\operatorname{N}}}
\newcommand{\lin}[1]{#1^\times / #1^{\times p}}
\newcommand{\cores}{\operatorname{Cores}}
\newcommand{\rad}{\operatorname{Rad}}
\newcommand{\E}{\mathscr{E}}

\newtheoremstyle{pedro}{}{}{\itshape}{}{\sc}{~--}{ }{\thmname{#1}\thmnumber{ #2}\thmnote{ (#3)}}

\newtheoremstyle{pedrodef}{}{}{}{}{\sc}{~--}{ }{\thmname{#1}\thmnumber{ #2}\thmnote{ (#3)}}

\title{Extensions of unipotent groups, Massey products and Galois theory}

\author{Pierre Guillot}
\address{
Pierre Guillot \\
Universit\'{e} de Strasbourg \& CNRS\\
Institut de Recherche Math\'{e}matique Avanc\'{e}e\\
UMR 7501, F-67000 Strasbourg, France}
\email{guillot@math.unistra.fr}
\author{J\'an Min\'a\v{c}}
\thanks{JM is partially supported by the Natural Sciences and Engineering Research Council of Canada (NSERC) grant R0370A01}
\address{
J\'an Min\'a\v{c} \\
Department of Mathematics\\
Western University\\
London, Ontario, N6A 5B7\\
Canada}
\email{minac@uwo.ca}


\renewcommand{\H}{\operatorname{H}}

\begin{document}

\begin{abstract}
  We study the vanishing of four-fold Massey products in mod~$p$ Galois cohomology. First, we describe a sufficient condition, which is simply expressed by the vanishing of some cup-products, in direct analogy with the work of Guillot, Min\'a\v{c} and Topaz for~$p=2$. For local fields with enough roots of unity, we prove that this sufficient condition is also necessary, and we ask whether this is a general fact.

  We provide a simple splitting variety, that is, a variety which has a rational point if and only if our sufficient condition is satisfied. It has rational points over local fields, and so, if it satisfies a local-global principle, then the Massey Vanishing conjecture holds for number fields with enough roots of unity.

  At the heart of the paper is the construction of  a finite group~$\tilde \U_5(\f_p)$,  which has~$\U_5(\f_p)$ as a quotient. Here~$\U_n(\f_p)$ is the group of unipotent~$n\times n$-matrices with entries in the field~$\f_p$ with~$p$ elements; it is classical that~$\U_{n+1}(\f_p)$ is intimately related to~$n$-fold Massey products. Although~$\tilde \U_5(\f_p)$ is much larger than~$\U_5(\f_p)$, its definition is very natural, and for our purposes, it is easier to study.
\end{abstract}
\maketitle

\section{Introduction}

In recent years, a lot of papers have been devoted to the investigation of {\em Massey products} in Galois cohomology. Recall that, if~$\Gamma $ is a profinite group, and if~$x_1$, $x_2$, $\ldots$, $x_n \in \h^1(\Gamma , \f_p)$, where~$p$ is a prime number, the Massey product of these classes is a certain subset 
\[ \langle x_1, \ldots, x_n \rangle \subset \h^2(\Gamma , \f_p) \, ,  \]
which may very well be empty. The {\em Massey Vanishing Conjecture} (see~\cite{MT17}, \cite{MT16}) essentially states that, when~$\Gamma = \Gal(\bar F/F)$ is the absolute Galois group of the field~$F$, all the higher Massey products are trivial. More precisely, the claim is that for all choices of~$1$-dimensional classes~$x_1$, $\ldots $, $x_n$, with~$n \ge 3$, either their Massey product is empty, or~$0 \in \langle x_1, \ldots, x_n \rangle$. In the latter case, one says that the Massey product {\em vanishes}, and so the conjecture asserts that non-empty Massey products, in Galois cohomology, all vanish.

The present work is a continuation of~\cite{previous}, by Guillot, Min\'a\v{c} and Topaz, and we refer the reader to this paper for more background, historical comments, and references. Here, let us simply state that the conjecture is known to hold for~$n=3$ (all fields, all~$p$), for local fields (all~$n$, all~$p$), while~\cite{previous}, together with its appendix by Wittenberg, settled the case~$n=4$, $p=2$, when~$F$ is a number field. On top of the references above, the reader may consult~\cite{hopkins}, \cite{EM-triple}, \cite{MT3}.   

We propose to examine the case~$n=4$, $p$ arbitrary, when~$F$ is a field containing a primitive~$p$-th root of unity. Here it is useful to recall the main result of~\cite{previous}. For classes~$x_1, x_2, x_3, x_4 \in \h^1(\Gal(\bar F/F), \f_2)$, where~$F$ is an arbitrary field of characteristic~$\ne 2$ and~$\bar F$ is a separable closure of~$F$, it is proved that~$0 \in \langle x_1, x_2, x_3, x_4 \rangle$ is equivalent to the vanishing of a certain collection of cup-products (just one cup-product in good cases). A statement with full details, generalized to all~$p$, is given below (Theorem~\ref{thm-main-intro}). Cup-products are much simpler than higher Massey products, and this description is surprisingly simple indeed. It allows to prove the Massey Vanishing Conjecture directly in a number of cases. Crucially, it also allows the description of a simple {\em splitting variety}, that is, a variety which possesses an~$F$-rational point if and only if~$0 \in \langle x_1, x_2, x_3, x_4 \rangle$. The appendix to~\cite{previous} shows that this variety satisfies a certain local-to-global principle, and thus the result for number fields is deduced from that for local fields.

Massey products of~$n$ classes in mod~$p$ cohomology are intimately related to the group~$\U_{n+1}(\f_p)$ of upper triangular $(n+1)\times (n+1)$-matrices with entries in~$\f_p$, with~$1$'s on the diagonal. (Just how this connection is made will be explained in the next section.) The arguments in~\cite{previous} thus depend on the structure of the group~$\U_5(\f_2)$, and they break for a general prime~$p$. Much of the present paper is dedicated to the construction of a certain group~$\tilde \U_5(\f_p)$, which for us is the ``right'' analog of~$\U_5(\f_2)$ for~$p$ odd. We should point out that~$\tilde \U_5(\f_p)$ embeds into~$\U_{2p+1}(\f_p)$, and this is a way of seeing that, for~$p=2$, we have~$\tilde \U_5(\f_2) = \U_5(\f_2)$. Explicit matrices are given in~$\U_{2p+1}(\f_p)$, which generate~$\tilde \U_5(\f_p)$;  the reader can have a glance at Proposition~\ref{prop-explicit-U5}. The arguments of~\cite{previous} can then be adapted to~$\tilde \U_5(\f_p)$, although the proofs are more delicate.

What is more, by construction the group~$\U_5(\f_p)$ is a quotient of~$\tilde \U_5(\f_p)$. As a result, we have initially a {\em sufficient} condition, in terms of cup-products, implying that~$0 \in \langle x_1, x_2, x_3, x_4 \rangle$. Let us be more precise. Let~$F$ be a field containing a primitive~$p$-th root of unity~$\dz$. In this situation there is an isomorphism 
\[ \h^1(\Gal(\bar F/F), \f_p) \cong \lin F \, .   \]
This allows us to speak of the Massey product $\langle a, b, c, d \rangle$ when~$a, b, c, d \in F^\times$. (Also, as is standard, we write cup-products~$(a,b)_F$, for~$a, b \in F^\times$.) The condition $0 \in \langle a, b, c, d \rangle$ translates into the existence of a Galois extension~$K/F$ with~$\Gal(K/F)$ isomorphic to a subgroup of~$\U_5(\f_p)$, such that~$F[\sq a, \sq b, \sq c, \sq d] \subset K$. More precisely, the extension~$K/F$ must be ``compatible with~$a, b, c, d$'', a condition we spell out in the next section. Similarly, when~$L/F$ is a Galois extension with~$\Gal(L/F)$ identified with a subgroup of~$\tilde \U_5(\f_p)$, we will naturally arrive at a compatibility condition of~$L/F$ with~$a, b, c, d$. When such a compatible~$L/F$ exists, we say that {\em the Massey product of~$a, b, c, d$ vanishes in the sense of~$\tilde \U_5(\f_p)$}; this implies that the Massey product vanishes in the usual sense, that is, $0 \in \langle a, b, c, d \rangle$. In fact, from Corollary~\ref{coro-repeated-massey} we will deduce the vanishing of a whole collection of~$k$-fold Massey products, all involving~$a, b, c, d$, for various values of~$k$ with~$2 \le k \le 2p$.

We will prove the following Theorem (see Theorem~\ref{thm-main} in the text).

\begin{thm} \label{thm-main-intro}
  Let~$F$ be a field with~$\dz \in F$, and let~$a, b, c, d \in F^\times$. For generic values of~$a, b, c, d$, the following assertions are equivalent.
  \begin{enumerate}
  \item The Massey product of~$a, b, c, d$ vanishes in the sense of~$\tilde \U_5(\f_p)$.
    
    \item One can find~$B \in F[\sq a]$ such that~$\N_{F[\sq a]/F}(B)= b f_1^p$ for some~$f_1 \in F^\times$, and~$C \in F[\sq d]$ such that~$\N_{F[\sq d]/F}(C)= c f_2^p$ for some~$f_2 \in F^\times$, with the property that for any~$\sigma \in \Gal(F[\sq a]/F)$ and any~$\tau \in \Gal(F[\sq d]/F)$, we have 
\[ (\sigma (B), \tau (C))_{F[\sq a, \sq d]} = 0 \, .   \]
\end{enumerate}
\end{thm}

Of course, Theorem~\ref{thm-main} spells out what ``generic values'' means. We point out that we first prove a purely group-theoretical result, Theorem~\ref{thm-maps-into-Un-tilde}, which is very general. The need to restrict (very mildly) the values of~$a, b, c, d$ emerges when we perform the translation from group theory to Galois theory.

For~$p=2$, one recognizes Theorem A from~\cite{previous} (and so a good deal of~~\cite{previous}, though not all, is contained in the present paper). 

Note also that the vanishing of a Massey product ``in the sense of~$G$'' can be given a definition for any group~$G$. We explore this in Remark~\ref{rmk-G-massey-products}. 

\[ \star \star \star  \]

Happily, we seem to have more than a sufficient condition. A phenomenon of {\em automatic Galois realization} is taking place. This is the name we give to results similar to the following, classical ones. When~$K/F$ is Galois with~$\Gal(K/F) \cong C_p$, and if we are willing to assume that~$p$ is odd and that~$F$ contains a primitive~$p^2$-th root of unity, then there exists automatically an extension~$L/F$ with~$\Gal(L/F) \cong C_{p^2}$ and~$K \subset L$ (this is a familiar exercise). In fact it is known, but a little harder to prove, that the existence of an extension~$K/F$ with~$\Gal(K/F) \cong C_p$ (with~$p$ odd) is enough to guarantee the existence of a Galois extension~$L/F$ with~$\Gal(L/F) \cong \Z_p$, see~\cite{whaples}; however, this time it is not claimed that~$K \subset L$. In the same vein, when~$\Gal(K/F) \cong \U_3(\f_p)$, there automatically exists~$L/F$ with~$\Gal(L/F) \cong C_p \wr C_p$ (wreath product of~$C_p$ with itself): we prove this in Proposition~\ref{prop-hilbert-symbol} (where~$C_p \wr C_p$ is written~$\f_p[\U_2] \rtimes \U_2$), but the result is well-known, in some form or other (see~\cite{automatic}). Again, in this second example, it is not claimed that~$K \subset L$.

A consequence of our Proposition~\ref{prop-local-fields} below is this:

\begin{thm} \label{thm-local-intro}
  Let~$F$ be a local field containing a primitive~$p^2$-th root of unity. Let~$a, b, c, d \in F^\times$, satisfying the genericity assumption from Theorem~\ref{thm-main}. Then the following conditions are equivalent.
  \begin{enumerate}
  \item[(i)] $\langle a, b, c, d \rangle$ is non-empty.
  \item[(ii)] $\langle a, b, c, d \rangle$ vanishes.
  \item[(iii)] $\langle a, b, c, d \rangle$ vanishes in the sense of~$\tilde \U_5(\f_p)$.
    \item[(iv)] $(a,b)_F = (b,c)_F= (c, d)_F = 0$.
\end{enumerate}
\end{thm}

Thus we are witnessing an automatic realization of a~$\tilde \U_5(\f_p)$-extension, given the existence of a~$\U_5(\f_p)$-extension. Keep in mind that~$\U_5(\f_p)$ has order~$p^{10}$, while $\tilde \U_5(\f_p)$ has order~$p^{p^2 + 2p+2}$.

Here is how to recover this from the contents of the paper. The equivalence of (i) and (ii) is the Massey Vanishing Conjecture for local fields, while the equivalence of (iv) and (ii) was known, see~\cite[Prop.\ 4.1]{localmasseycup}. We have already pointed out that (iii) implies (ii), simply because~$\U_5$ is an appropriate quotient of~$\tilde \U_5$. The point made here is mostly that (iv) implies (iii); for this, use Proposition~\ref{prop-local-fields}, which shows that (iv) implies condition (2) from Theorem~\ref{thm-main-intro} above, and so by that Theorem we have the present condition (iii) (this is where we need~$a, b, c, d$ to be ``generic'').

It is a fascinating question to ask whether this extends beyond local fields. One is allowed to hope, at least, that Theorem~\ref{thm-local-intro} is also valid for number fields. We justify this enthusiasm with the work in the last section of this paper, where we produce a splitting variety for the vanishing of the Massey product of four elements~$a, b, c, d$ in the sense of~$\tilde \U_5(\f_p)$. It generalizes the construction of~\cite{previous}, from~$p=2$ to general~$p$. The variety is very simple geometrically, and one is tempted to believe that it should always satisfy a local-to-global principle for the existence of~$F$-rational points, as it does when~$p=2$, by the appendix to~\cite{previous}. There are considerable technical obstacles to overcome, to be sure, in order to prove the validity of such a principle. Still, we note that Theorem~\ref{thm-local-intro} would then hold for number fields (presumably with some restrictions on~$a, b, c, d$). Also the Massey Vanishing Conjecture would be proved for~$n=4$, $p$ arbitrary, and~$F$ a number field containing a primitive~$p^2$-th root of unity.

\[ \star \star \star  \]

The paper is organized as follows. Section~\ref{sec-prelim} contains various preliminaries about the groups~$\U_n(\f_p)$, explains their relationship with Massey products, and sets up the notation. The next four sections are group-theoretical in nature: we construct the group~$\tilde \U_n(\f_p)$, of which~$\U_n(\f_p)$ is quotient, whenever~$n$ is odd, and we prove a Theorem about maps from profinite groups into~$\tilde \U_n(\f_p)$. (We mostly care about the case~$n=5$, but the general case is no harder, and can be illuminating.) In Section~\ref{sec-translation}, we proceed to translate the group-theory into the langage of fields and Galois theory, and this Section culminates with a proof of Theorem~\ref{thm-main-intro}. Local fields are dealt with in Section ~\ref{sec-local-fields}. The last Section presents the construction of the splitting variety.

\bigskip

\noindent {\em Acknowledgements.} It is a pleasure to thank Nguyên Duy Tân for stimulating conversations, and for his interest, insight and encouragement. We are also grateful to participants of the BIRS workshop on Nilpotent Fundamental Groups in June 2017 for their interest in our work, in particular M.\ Florence, E.\ Matzri, F.\ Pop, A.\ Topaz,  K.\ Wickelgren and O.\ Wittenberg.

\section{Notation \& Preliminaries} \label{sec-prelim}

\subsection{Unipotent matrices}

We write~$\U_n(\f_p)$ for the group of unipotent~$n \times n$-matrices, with entries in the field~$\f_p$ having~$p$ elements, where~$p$ is a prime. Recall that a matrix is ``unipotent'' when it is upper-triangular, with~$1$'s on the diagonal. Very often we will write simply~$\U_n$ instead of~$\U_n(\f_p)$. We first collect the most basic facts about~$\U_n$, before explaining the relationship between this group and Massey products in mod~$p$ cohomology.

When~$g \in \U_n$, or more generally when~$g$ is a matrix, we write~$g_{ij}$ for the coefficient on the~$i$-th row, in the~$j$-th column of~$g$.

The centre~$\sZ(\U_n(\f_p))$ is isomorphic to~$\f_p$, and generated by~$I + g$, where~$I$ is the identity matrix, and the only non-zero coefficient of the matrix~$g$ is~$g_{1n} = 1$ (that is, in the top-right corner). We put~$\bar \U_n(\f_p) := \U_n(\f_p) / \sZ(\U_n(\f_p))$, also denoted simply by~$\bar \U_n$.

We write 
\[ s_i \colon \U_n(\f_p) \longrightarrow \f_p  \]
for the homomorphism~$g \mapsto g_{i,i+1}$, where~$1 \le i < n$. These can be combined into 
\[ (s_1, \ldots, s_{n-1}) \colon \U_n(\f_p) \longrightarrow \f_p^{n-1} \, ,   \]
which factors to give an isomorphism 
\[ \U_n(\f_p) / \Phi (\U_n(\f_p)) \longrightarrow \f_p^{n-1} \, ,   \]
where $\Phi (\U_n(\f_p))$ is the Frattini subgroup of~$\U_n$. Thus we see that~$\Phi (\U_n)$ is comprised of those matrices having zero entries on the ``near diagonal''.

Next, for~$1 \le i < n$, we put~$\sigma_i = I + g$ where~$g_{i,i+1} = 1$ and~$g_{ij} = 0$ otherwise. The elements~$\sigma_1, \ldots, \sigma_{n-1}$ are generators for~$\U_n$, and if we write~$\bar \sigma_i$ for the image of~$\sigma_i$ in~$\U_n / \phi (\U_n)$, then $\bar \sigma_1, \ldots, \bar \sigma_{n-1}$ correspond to the canonical basis of~$\f_p^{n-1}$ under the above isomorphism.

If we view~$s_i$ as an element of~$\h^1(\U_n, \f_p)$, then a crucial remark is the vanishing of the cup-product
\[ s_i \cupp s_{i+1} = 0 \in \h^2(\U_n, \f_p) \, ,   \]
when~$n \ge 3$ and~$i < n - 1$; also, if we view~$s_i$ as an element of~$\h^1(\bar \U_n, \f_p)$, as we clearly can, then 
\[ s_i \cupp s_{i+1} = 0 \in \h^2(\bar \U_n, \f_p) \, ,  \]
this time with the restriction~$n \ge 4$. Indeed, to prove both statements, use the homomorphism 
\[ \U_n \longrightarrow \bar \U_n \longrightarrow \U_3  \]
which deletes the first~$i-1$ rows and the~$n-2-i$ rightmost columns; this reduces the problem to showing 
\[ s_1 \cupp s_2 = 0 \in \h^2(\U_3, \f_p) \, .   \]
And indeed, the cohomology class of the extension 
\[ 0 \longrightarrow \f_p \cong \sZ(\U_3) \longrightarrow \U_3 \longrightarrow \f_p \times \f_p \longrightarrow 0  \]
is just~$s_1 \cupp s_2$, as a matrix multiplication readily shows. (When~$p=2$, this is the familiar statement that the dihedral group of order~$8$ is described as an extension of~$\f_2^2$ by a central~$\f_2$, with the cohomology class being a cup-product.)

\subsection{Massey products}

Now, let~$\Gamma $ be any profinite group. Let~$\chi_1, \ldots, \chi_n \in \h^1(\Gamma, \f_p)$ be given, with~$n \ge 2$. We say that a (continuous) homomorphism~$\phi \colon \Gamma \to \bar \U_{n+1}(\f_p)$ is {\em compatible} with~$\chi_1, \ldots, \chi_n$ when~$s_i  \circ \phi  = \chi_i$ for all indices~$i$. The {\em Massey product} of these elements is 
\[ \langle \chi_1, \ldots, \chi_n \rangle := \{ \phi^*(\alpha ) : \phi ~\textnormal{compatible with}~\chi_1, \ldots, \chi_n \} \subset \h^2(\Gamma , \f_p) \, ,   \]
where~$\alpha \in \h^2(\bar \U_{n+1}, \f_p)$ is the cohomology class of the extension 
\[ 0 \longrightarrow \f_p \cong \sZ(\U_{n+1}) \longrightarrow \U_{n+1} \longrightarrow \bar \U_{n+1} \longrightarrow 1 \, .   \]
The Massey product may very well be empty. When it is not, it is customary to say that the Massey product is ``defined''.

We shall be particulary interested in situations when~$0 \in \langle \chi_1, \ldots, \chi_n \rangle$, in which case we say that the Massey products {\em vanishes}. (We also write that~$\langle \chi_1, \ldots, \chi_n \rangle$ vanishes.) This happens precisely when there is a~$\phi \colon \Gamma \to \bar \U_{n+1}$ which is compatible with~$\chi_1, \ldots, \chi_n$ and which can be lifted to a homomorphism~$\Gamma \to \U_{n+1}$. Put more directly, the Massey product of~$\chi_1, \ldots, \chi_n$ vanishes if and only if there is a (continuous) homomorphism~$\psi \colon \Gamma \to \U_{n+1}$ such that~$s_i \circ \psi = \chi_i$ for all~$i$. Again, such a $\psi$ is called {\em compatible} with~$\chi_1, \ldots, \chi_n$.

As a particular case, note that when~$n=2$, there is just one homomorphism~$\phi \colon \Gamma \to \bar \U_3(\f_p) = \f_p \times \f_p $ compatible with~$\chi_1$ and~$\chi_2$, namely~$(\chi_1, \chi_2)$. It follows that 
\[ \langle \chi_1, \chi_2 \rangle = \{ \chi_1 \cupp \chi_2 \} \, .   \]
It should be noted that our definition of Massey products is based on the work of Dwyer in~\cite{dwyer}, and is not always the standard definition given in the literature.

\begin{rmk} \label{rmk-G-massey-products}
  It is of course possible to replace~$\U_n$ and~$\bar \U_n$ by other groups, and obtain new Massey products. Let us briefly give a few details, since what we do in the paper is related to that idea. Let~$G$ be any finite group, and suppose it is equipped with distinguished cohomology classes~$s_1, \cdots, s_n \in \h^1(G, \f_p)$ (these will not appear in the notation). For any profinite group~$\Gamma $ and classes~$\chi_1, \cdots, \chi_n \in \h^1(\Gamma, \f_p)$, let us say that {\em the Massey product of~$\chi_1, \ldots , \chi_n$ vanishes in the sense of~$G$} when there exists~$\psi \colon \Gamma \to G$ with~$\psi^*(s_i) = \chi_i$ (or~$s_i \circ \psi = \chi_i$ if you prefer). Note that ``Dwyer-Massey products'' is perhaps a better name.

  Next, suppose~$A$ is a central, abelian subgroup of~$G$, with the property that each~$s_i$ vanishes on~$A$, and so can be seen as an element of~$\h^1(G/A, \f_p)$. Then we put 
\[ \langle \chi_1, \ldots, \chi_n \rangle_{G, A} := \{ \phi^*(\alpha ) ~|~ \phi \colon \Gamma \to G/A ~\textnormal{with}~ \phi^*(s_i) = \chi_i \} \subset \h^2(\Gamma, A) \, ,   \]
where~$\alpha \in \h^2(G/A, A)$ is the class of the extension 
\[ 0 \longrightarrow A \longrightarrow G \longrightarrow G/A \longrightarrow 1 \, .   \]
We call~$\langle \chi_1, \ldots, \chi_n \rangle_{G, A}$ the {\em Massey product of~$\chi_1, \ldots, \chi_n$ with respect to~$G$ and~$A$}. This set is non-empty when ``the Massey product of~$\chi_1, \ldots, \chi_n$ vanishes in the sense of~$G/A$''; and it contains~$0$ when ``the Massey product of~$\chi_1, \ldots, \chi_n$ vanishes in the sense of~$G$''. When~$G= \U_{n+1}(\f_p)$ and~$A= \sZ(G)$, we recover the usual Massey products.

As explained in the Introduction, this paper is mostly dedicated to the vanishing of cohomology classes in the sense of a certain group~$\tilde \U_5$, to be constructed. The vocabulary just described will not be used, however; with this remark we merely wanted to point out the possible generalizations that can be envisaged.
\end{rmk}

\subsection{Galois theory}

 Most fields  will be assumed to contain a primitive~$p$-th root of unity~$\dz$, where~$p$ is our usual fixed prime. For brevity, we will indicate this by writing merely~``$\dz \in F$''.

Let~$\bar F$ be a separable closure of~$F$. We systematically write $G_F:= \Gal(\bar F/F)$ for the absolute Galois group of~$F$. We also write~$\h^*(F, \f_p)$ for~$\h^*(G_F, \f_p)$. When~$\dz \in F$, we have an isomorphism 
\[ \h^1(F, \f_p) \cong \lin F \, .  \]
The homomorphism~$\Gal(\bar F/F) \to \f_p$ corresponding to~$a \in F^\times$ will be denoted by~$\chi _a$. We also write 
\[ \langle a_1, \ldots, a_n \rangle := \langle \chi_{a_1}, \ldots, \chi_{a_n} \rangle \subset \h^2(F, \f_p) \, ,   \]
and call this set the Massey product of~$a_1, \ldots, a_n \in F^\times$. Accordingly, this Massey product can be said to be ``defined'' (= non-empty), or to vanish, as the case may be. A homomorphism~$G_F \to \U_{n+1}(\f_p)$ (or $\bar \U_{n+1}$) will be said to be compatible with~$a_1, \ldots, a_n$ when it is compatible with~$\chi_{a_1}, \ldots, \chi_{a_n}$. For convenience, when~$K/F$ is a Galois extension with~$\Gal(K/F)$ explicitly identified with a subgroup of either~$\bar \U_{n+1}(\f_p)$ or~$\U_{n+1}(\f_p)$, we will also speak of~$K/F$ being compatible with~$a_1, \ldots, a_n$, in the obvious sense.

Also note that we write 
\[ (a,b)_F := \chi_a \cupp \chi_b \, ,   \]
as is classical (we often abbreviate to~$(a,b)$). Thus, when~$\langle a_1, \ldots, a_n \rangle$ is defined, with~$n \ge 3$, we have~$(a_i, a_{i+1}) = 0$.

To be a little more concrete, suppose that~$\Gal(K/F) \subset \U_{n+1}$ (or~$\bar \U_{n+1}$) and that~$K/F$ is compatible with~$a_1, \ldots, a_n$. The kernel of~$\chi_{a_i}$ corresponds to the field $F[\sq a_i]$, in the Galois correspondence. The ``compatibility'' of~$K/F$ implies then that~$F[\sq a_i] \subset K$, and so~$F[\sq a_1, \ldots, \sq a_n] \subset K$. If~$\Gal(K/F)$ is all of~$\U_{n+1}$, whose Frattini quotient we have described, we see that $F[\sq a_1, \ldots, \sq a_n]$ is the largest $p$-Kummer extension of~$F$ contained in~$K$.

The {\em Massey vanishing conjecture} stipulates that, for any prime~$p$ and for any field~$F$ with~$\dz \in F$, and any~$a_1, \ldots, a_n \in F^\times$, with~$n \ge 3$, the Massey product~$\langle a_1, \ldots, a_n \rangle$ vanishes whenever it is non-empty. (The conjecture in fact extends, in the obvious fashion, to Massey products in fields which do not necessarily contain enough~$p$-th roots of unity, but we prefer to speak of Massey products of elements of~$F^\times$.)

\subsection{A lemma from representation theory}
We frequently use the following classical Lemma:

\begin{lem} \label{lem-usual-lemma-rep-theory}
Let~$G$ be a~$p$-group, and let~$V$ be an~$\f_p[G]$-module, where~$p$ is a prime. Suppose~$v \in V$ satisfies~$\sum_{g \in G} g \cdot v \ne 0$. Then the~$\f_p[G]$-module spanned by~$v$ is free.
\end{lem}

\begin{proof}[Sketch]
  Recall that any non-zero~$\f_p[G]$-module contains a non-zero vector which is fixed by~$G$, whenever~$G$ is a~$p$-group. This follows from the Sylow Theorems, as~$\operatorname{GL}_n(\f_p)$ admits, as~$p$-Sylow subgroup, the group~$\U_n(\f_p)$ described above; its elements fix the first vector in the canonical basis of~$\f_p^n$.

  Next, apply this to the kernel of~$\pi \colon \f_p[G] \to V$, mapping~$1$ to~$v$; since~$\sum_{g \in G} g \not\in \ker(\pi)$, we see that~$\ker(\pi) = \{ 0 \}$.
\end{proof}

\section{The module~$S_n$} \label{sec-modS}

For convenience, we include an overview of the next four sections. We start with the description of a certain subgroup~$S_n$ of~$\U_n = \U_n(\f_p)$ (with~$p$ fixed throughout). When~$n=2m+1$, we have an exact sequence 
\[ 0 \longrightarrow S_n \longrightarrow  \U_n \longrightarrow \U_{m+1} \times \U_{m+1} \longrightarrow 1 \, , \tag{$\dagger$}   \]
which we wish to understand. We describe~$S_n$ as the tensor product~$V_m \otimes V_m^*$ where~$V_m$ is a certain~$\U_{m+1}$-module. In Section~\ref{sec-some-1-cocycles}, we describe~$\U_{m+1}$ as a semidirect product, in two different ways, which affords two 1-cocycles~$\phi_1 \in \h^1(\U_{m+1}, V_m)$ and~$\phi_2 \in \h^1(\U_{m+1}, V_m^*)$. In Section~\ref{sec-the-cup-product}, we show that the cup-product $\phi_1 \cupp \phi_2$ is none other than the class of the extension ($\dagger$). In a sense, the problem at hand has moved down from~$\h^2$ to~$\h^1$, which is much easier to deal with.

In Section~\ref{sec-wreath-extensions}, we use some lifts~$\tilde \phi_1$ and~$\tilde \phi_2$ of~$\phi_1$ and~$\phi_2$ respectively, to some extension groups of~$\U_m$ written~$\U_m ^{(i)}$ for~$i= 1, 2$. We use these to form the product $\tilde \phi_1 \cupp \tilde \phi_2$, which defines an extension of~$\U_m ^{(1)} \times \U_m ^{(2)}$, called~$\tilde \U_n$. The construction is so arranged that maps~$\Gamma \to \tilde \U_n$, where~$\Gamma $ is any profinite group, are extremely simple to understand. (More precisely, a certain lifting problem comes completely under control.) This is Theorem~\ref{thm-maps-into-Un-tilde}, the culmination of the ``group-theoretic part'' of the paper.  

It is  helpful to keep in mind that we mostly care about the case~$n=5$, throughout the paper. However, we include the general~$n=2m+1$ case (and even some information about~$n=2m$) because it is no harder (and in fact, clearer). There is nothing special about the number~$5$ at this point, and it would be misleading to give this impression. One thing we only prove for~$n=5$, however, is Proposition~\ref{prop-explicit-U5}, which provides an alternative description of~$\tilde \U_5(\f_p)$ as a group of matrices.

\bigskip

Let~$n \ge 2$ be given, and let~$m$ be defined by~$n= 2m$ if~$n$ is even, and~$n= 2m+1$ otherwise. Of paramount importance to us is the subgroup~$S_n \subset \U_n$, where the letter ``S'' is for ``square'', comprised of the elements whose non-zero entries are in the~$m \times m$-corner on the top-right (and on the diagonal), that is 
\[ S_{2m} := \{ g \in \U_{2m}(\f_p) : g_{ij}= 0 ~\textnormal{for}~ i<j\le m ~\textnormal{or}~ m \le i < j \} \, ,  \]
while 
\[ S_{2m+1} := \{ g \in \U_{2m+1}(\f_p) : g_{ij}= 0 ~\textnormal{for}~ i<j\le m+1 ~\textnormal{or}~ m+1 \le i < j \} \, .  \]
For example, the elements of~$S_5$ have the shape 
\[ \left(\begin{array}{ccccc}
1  & 0 & 0 & * &*   \\
0  & 1 & 0 & * & *  \\
0  & 0 & 1 & 0 & 0   \\
0  & 0 & 0 & 1 & 0  \\
0  & 0 & 0 & 0 & 1   
\end{array}\right) \, .   \]
%
% copy and paste this for an element of U5:
%% \[ \left(\begin{array}{ccccc}
%% 1  &  &  &  &   \\
%% 0  & 1 &  &  &   \\
%% 0  & 0 & 1 &  &   \\
%% 0  & 0 & 0 & 1 &   \\
%% 0  & 0 & 0 & 0 & 1   
%% \end{array}\right)  \]
%The case~$n=5$ is the most important one, and our examples will usually be of this size.

We will think of the elements of~$\U_n$ as linear transformations of~$V_n := \f_p^n$, the latter being identified with the space of column matrices (the action being on the left); in fact we identify, once and for all, the matrices of size~$n \times n$ with the endomorphisms of~$V_n$. We will sometimes call~$V_n$ the {\em natural} $\U_n$-module.  We write~$e_1$, $\ldots$, $e_n$ for the canonical basis of~$V_n$, and we put 
\[ U_n = \langle e_1, \ldots, e_m \rangle \, , \qquad W_n = \langle e_{n-m +1}, e_{n-m+2}, \ldots, e_n \rangle \, .   \]
When~$n= 2m$, we have thus 
\[ V_n = U_n \oplus W_n \, ,   \]
while in the case~$n= 2m+1$ we have 
\[ V_n = U_n \oplus \langle e_{m+1} \rangle \oplus W_n \, .   \]
In this situation, we put~$U_n^+ = U_n \oplus \langle e_{m+1} \rangle$ and~$W_n^+ = W_n \oplus \langle e_{m+1} \rangle$, so that~$V_n = U_n^+ \oplus W_n = U_n \oplus W_n^+$, and $U_n^+ \cap W_n^+ = \langle e_{m+1} \rangle$).

In order to have a uniform notation for all~$n$, we decide to put~$U_n^+ = U_n$ when~$n=2m$, and also~$W_n^+ = W_n$.

We can now describe~$S_n$ as: 
\[ S_n = \{ g= 1+h  : h(U_n^+) =0 \, , \quad h(W_n) \subset U_n \} \, .   \]
Here it is understood that~$1$ is the identity of~$V_n$, and $h$ is a linear map $V_n \to V_n$. Note that, if~$1+h_1$ and~$1+h_2$ are elements of~$S_n$, then~$h_1 h_2 = 0$, and so 
\[ (1 + h_1)(1+h_2) = 1 + h_1 + h_2 \, .   \]
We have proved:

\begin{lem}
The group~$S_n$ is isomorphic to~$\Hom(W_m, U_m)$, {\em via} $1+h \mapsto h$. In particular, it is an elementary abelian~$p$-group. \hfill $\square$
\end{lem}

The subgroup~$S_n$ is normal in~$\U_n$: indeed $g(1+h)g^{-1} = 1 + g h g^{-1}$, and for~$g \in \U_n$, the spaces~$U_n$ and~$U_n^+$ are preserved by~$g$, showing that~$ghg^{-1}$ satisfies the same conditions as~$h$ does. Since~$S_n$ is abelian, the conjugation action factors through~$G:= \U_n(\f_p) / S_n$, and we wish to describe it. 

When~$n=2m$, we have~$G \cong \U_m \times \U_m$, clearly, and we write~$G= G_1 \times G_2$ accordingly. In this case, the group~$\U_n$ is in fact a semi-direct product: $\U_n = S_n \rtimes G$, where we use the section $G \cong \U_m \times \U_m \to \U_n$ which fills the top-right~$m \times m$-corner with~$0$'s. Thus, we can immediately see the elements of~$G_1$ as linear maps of~$U_m$, and the elements of~$G_2$ as linear maps of~$W_m$. Alternatively, this amounts to identifying the natural module~$V_m$, on which~$\U_m$ certainly acts, with~$U_m$ or~$W_m$, using our canonical bases.

When~$n=2m+1$, things are a little more delicate. We have~$G \cong \U_{m+1} \times \U_{m+1}$, allowing us to write~$G = G_1 \times G_2$ again, but the extension 
\[ 0 \longrightarrow S_n \longrightarrow \U_n(\f_p) \longrightarrow \U_{m+1} \times \U_{m+1} \longrightarrow 1 \, ,   \]
to which this paper is mostly devoted, is not split. For our initial purposes however, this is not a real problem. We do have a section~$G_1 \cong \U_{m+1} \to \U_n$ which copies the given~$(m+1) \times (m+1)$ matrix in the top-left corner, and uses entries from the identity matrix otherwise. Likewise, we have a section~$G_2 \cong U_{m+1} \to \U_n$, sticking the matrix in the bottom-right corner. This allows us to see the elements of~$G_1$ as linear transformations of~$U_n^+$, and those of~$G_2$ as linear transformations of~$W_n^+$.

To get rid of the ``+'', we use the following remarks, on which we will expand below. We consider the map 
\[ \pi = \pi_m \colon \U_{m+1} \longrightarrow \U_m  \]
which deletes the rightmost column and the bottom row. Likewise, consider the homomorphism 
\[ \pi' = \pi_m' \colon \U_{m+1} \longrightarrow \U_m  \]
which deletes the top row and the leftmost column. We combine them into a map 
\[ (\pi, \pi') \colon \U_{m+1} \times \U_{m+1} \longrightarrow \U_m \times \U_m \, .   \]
In turn, there is an obvious section $\U_m \times \U_m \to \U_{2m+1}$. So, {\em via} the map~$(\pi, \pi')$, we can see the elements of~$G_1$ acting on~$U_m$, and those of~$G_2$ acting on~$W_m$ -- exactly as in the~$n=2m$ case.

We prove now that the action is ``the obvious one'':

\begin{prop} \label{prop-S-is-tensor-product}
The action of~$G \cong G_1 \times G_2$ on~$S_n \cong \Hom(W_m, U_m)$ by conjugation is the natural one. In more details, suppose~$n=2m$ first. Then~$(g_1, g_2) \in G_1 \times G_2 \cong \U_m \times \U_m$ acts on~$h \in \Hom(W_m, U_m)$ by 
\[ (g_1, g_2) \cdot h (w) = g_1 \cdot h(g_2^{-1} \cdot w) \, .   \]
If we suppose that~$n=2m+1$, then~$(g_1, g_2) \in G_1 \times G_2 \cong \U_{m+1} \times \U_{m+1}$ acts on~$h \in \Hom(W_m, U_m)$ by 
\[ (g_1, g_2) \cdot h (w) = \pi(g_1) \cdot h(\pi'(g_2)^{-1} \cdot w) \, .   \]

In other words, the module~$S_n$ is isomorphic to~$V_m \otimes V_m^{*}$, with its natural~$\U_m \times \U_m$-action, pulled back to~$\U_{m+1} \times \U_{m+1}$ {\em via} $(\pi, \pi')$ in case~$n=2m+1$.
\end{prop}

\begin{proof}
  We give the argument in the case~$n=2m+1$, which is more important and (slightly) more difficult. Suppose~$f= 1 + h \in S_n$. Consider first the action of~$(g, 1) \in G_1 \times G_2$. We write also~$g$ for the lift in~$\U_n$ described above (using~$g \in \U_{m+1}$ in the top-left corner, and entries from the identity otherwise). We have~$(g, 1) \cdot f = g f g^{-1} = 1 + g h g^{-1}$.

  Since~$g^{-1}$ is the identity on~$W_n$, we have for~$w \in W_n$: 
\[ ghg^{-1}(w) = g(h(w)) \, .   \]
However, $h(w) \in U_m$ by definition of~$S_n$, so~$g(h(w))=\pi(g) \cdot h(w)$. When we use the isomorphism $S_n \cong \Hom(W_n, U_n)$, which maps $f=1+h$ to~$h$, the action of~$(g,1)$ is thus as the proposition predicts.

Now we consider the action of~$(1, g)$, and again the letter~$g$ will denote simultaneously an element of~$\U_{m+1}$ and a lift in~$\U_n$, this time using the original entries in the bottom-right corner. Again the action of~$(1, g)$ on~$f = 1+h$ is given by~$g f g^{-1} = 1 + g h g^{-1}$. If~$w \in W_n$ we certainly have 
\[ g h g^{-1}(w) = h g^{-1}(w) \, ,   \]
since~$h$ takes its values in~$U_n$, and~$g$ is the identity there. On the other hand, we have~$g^{-1}(w) \in W_n^+$. However, it is assumed that~$h$ is~$0$ on all of~$U_n^+$, so on~$U_n^+ \cap W_n^+ = \langle e_{m+1} \rangle$, so~$h(g^{-1}(w))$ only depends on the image of~$g^{-1}(w)$ under the projection
\[ W_n^+ \longrightarrow W_n  \]
with kernel $\langle e_{m+1} \rangle$. Thus $h(g^{-1}(w)) = h(\pi'(g)^{-1} \cdot w)$. The action is as promised.

The last statement follows by identifying~$U_m \cong W_m \cong V_m$ using the canonical bases at our disposal, and appealing to the well-known isomorphism $\Hom(V_m, V_m) \cong V_m \otimes V_m^*$, where~$V_m^* = \Hom(V_m, \f_p)$.
\end{proof}

\begin{coro}
The centralizer of~$S_n$ in~$\U_n(\f_p)$ is 
\[ \sC(S_n) = \{ 1 + h : h(U_n) = 0 \, , \quad h(W_n) \subset U_n^+ \} \, .   \]
So when~$n$ is even, we have~$\sC(S_n) = S_n$, and when~$n$ is odd, the non-zero entries of the elements of~$\sC(S_n)$ are in the top-right $(m+1) \times (m+1)$-corner (and on the diagonal).
\end{coro}

For example when~$n=5$, the elements of~$\sC(S_5)$ have the shape
\[ \left(\begin{array}{ccccc}
1  & 0 &*  &*  &*   \\
0  & 1 &*  &*  &*   \\
0  & 0 & 1 &*  & *  \\
0  & 0 & 0 & 1 & 0  \\
0  & 0 & 0 & 0 & 1   
\end{array}\right) \, .  \]

\begin{proof}
The action of~$\U_m \times \U_m$ on~$V_m \otimes V_m^*$ is clearly faithful, and the case~$n=2m$ follows from this (and the proposition). When~$n=2m+1$, the kernel of the action of~$G_1 \times G_2 \cong \U_{m+1} \times \U_{m+1}$ on~$V_m \otimes V_m^*$ is, this time, $\ker(\pi, \pi')$. In fact, if~$N = \ker(\pi)$ and~$N' = \ker(\pi')$, then~$\ker(\pi, \pi') = N_1 \times N_2'$, in obvious notation. The inverse image of~$\ker(\pi, \pi')$ under the quotient map~$\U_n \longrightarrow \U_n/S_n$ is~$\sC(S_n)$, and this gives the corollary.
\end{proof}

\section{Some $1$-cocycles} \label{sec-some-1-cocycles}

We have introduced the projection 
\[ \pi \colon \U_{m+1} \longrightarrow \U_m \, .   \]
We put~$N= \ker(\pi)$, which consists of those elements of~$\U_{m+1}(\f_p)$ whose non-zero coefficients are in the last column (and on the diagonal). For~$m=2$ (still thinking of~$n= 2m+1 = 5$) these are the elements of the form 
\[ \left(\begin{array}{ccc}
  1 & 0 & * \\
  0 & 1 & * \\
  0 & 0 & 1
\end{array}\right) \, .   \]

\begin{lem} \label{lem-Um-is-semidirect}
\begin{enumerate}
\item $N$ is elementary abelian, of order~$p^{m}$. More precisely, the map $N \to V_m$, mapping an element of~$N$ to its last column with the last entry deleted, is a linear isomorphism.
\item The group~$\U_{m+1}$ is a semi-direct product: $\U_{m+1} = N \rtimes \U_m$. The section~$\U_m \to \U_{m+1}$ inserts the matrix in the top-left corner.
  \item The action of~$\U_{m}$ on~$N$ by conjugation agrees with the natural action of~$\U_m$ on~$V_m$, when we identify the vector spaces~$N$ and~$V_m$ as in (1).
\end{enumerate}
\end{lem}

\begin{proof}
Easy.
\end{proof}

Let us view~$V_m \cong N$ as a~$\U_{m+1}$-module (using~$\pi$, or equivalently using the conjugation action on~$N \subset \U_{m+1}$). What we have gained here is a~$1$-cocycle for~$\U_{m+1}$, with values in~$N$, in virtue of the well-known:

\begin{lem} \label{lem-semidirect-gives-cocycle}
Let~$G$ be a finite group which is a semi-direct product~$G= N \rtimes H$, where~$N$ is abelian. Any element~$g \in G$ may thus be written uniquely~$g = n_g h_g$ with~$n_g \in N$, $h_g \in H$. The map $\phi \colon G \to N$ defined by~$\phi (g) = n_g$ is then a~$1$-cocycle, that is, it satisfies 
\[ \phi (g_1 g_2) = \phi (g_1) + g_1 \cdot \phi (g_2) \, ,   \]
where the action employed is the conjugation.

Moreover, the restriction of~$\phi $ to~$N$ is the identity, and the restriction of~$\phi $ to~$H$ is trivial.
\end{lem}

In our case, with~$G= \U_3$, we have 
\[ \left(\begin{array}{ccc}
  1 & 0 & x \\
  0 & 1 & y \\
  0 & 0 & 1
\end{array}\right) \left(\begin{array}{ccc}
  1 & a & 0 \\
  0 & 1 & 0 \\
  0 & 0 & 1
\end{array}\right) = \left(\begin{array}{ccc}
  1 & a & x \\
  0 & 1 & y \\
  0 & 0 & 1
\end{array}\right) \, ,  \]
so that 
\[ \phi \left(\begin{array}{ccc}
  1 & a & x \\
  0 & 1 & y \\
  0 & 0 & 1
\end{array}\right) = \left(\begin{array}{c}
x \\ y
\end{array}\right) \in V_2 \cong N \, .   \]
In general, the lemma gives the $1$-cocycle~$\phi \in \h^1(\U_{m+1}, V_m)$, which maps an element of~$\U_{m+1}$ its last column, with the last entry deleted.

We proceed to analyse similarly the group~$N'= \ker(\pi')$; for~$m= 3$ this is the group of matrices of the form 
\[ \left(\begin{array}{ccc}
  1 & * & * \\
  0 & 1 & 0 \\
  0 & 0 & 1
\end{array}\right) \, .   \]
Initially, the situation is merely dual to the previous one. Recall that we write~$V_m$ for the vector space~$\f_p^m$ of column-matrices; we will now write~$V_m^*$ for the space of row-matrices (of dimension~$m$ over~$\f_p$).

\begin{lem} \label{lem-Um-is-semidirect-dual}
\begin{enumerate}
\item $N'$ is elementary abelian, of order~$p^{m}$. More precisely, the map $N' \to V_m^*$, mapping an element of~$N'$ to its first row with the first entry deleted, is a linear isomorphism.
\item The group~$\U_{m+1}$ is a semi-direct product: $\U_{m+1} = N' \rtimes \U_m$. The section~$\U_m \to \U_{m+1}$ inserts the matrix in the bottom-right corner.
  \item The action of~$\U_{m}$ on~$N'$ by conjugation agrees with the natural action of~$\U_m$ on~$V_m^*$, when we identify the vector spaces~$N'$ and~$V_m^*$ as in (1).
\end{enumerate}
\end{lem}

Now, we can decide to apply Lemma~\ref{lem-semidirect-gives-cocycle}, but the resulting cocycle will have a more complicated explicit form. This is illustrated with~$m=3$ by the following computation:
\[ \left(\begin{array}{ccc}
  1 & 0 & 0 \\
  0 & 1 & a \\
  0 & 0 & 1
\end{array}\right) \left(\begin{array}{ccc}
  1 & x & y \\
  0 & 1 & 0 \\
  0 & 0 & 1
\end{array}\right) = \left(\begin{array}{ccc}
  1 & x & y \\
  0 & 1 & a \\
  0 & 0 & 1
\end{array}\right) \, .  \]
This shows that an element~$g \in \U_3$ can be very simply decomposed as~$g = hn$ with~$n \in N'$ and~$h \in \U_2$; note how this is reversed. From this of course we may write~$g = hnh^{-1}h = (h \cdot n) h$, so the cocycle produced by Lemma~\ref{lem-semidirect-gives-cocycle} is~$g \mapsto h \cdot n = g \cdot n$, resulting in a formula which is a little less pleasant than that for~$\phi $. In general, we have the following lemma, which is not as well-known as Lemma~\ref{lem-semidirect-gives-cocycle}, but whose proof is just as straightforward.

\begin{lem}
Let~$G$ be a finite group which is a semi-direct product~$G= N \rtimes H$, where~$N$ is abelian. Any element~$g \in G$ may thus be written uniquely~$g = h_gn_g$ with~$n_g \in N$, $h_g \in H$. The map $\psi \colon G \to N$ defined by~$\phi (g) = n_g$ satisfies 
\[ \psi (g_1 g_2) = g_2^{-1}\psi (g_1) +  \psi (g_2) \, ,   \]
where the action employed is the conjugation. The map~$g \mapsto g \cdot \psi (g)$ is a $1$-cocycle, and indeed it is that produced by Lemma~\ref{lem-semidirect-gives-cocycle}. Also, $g \psi (g) = - \psi (g^{-1})$ for all~$g \in G$. 

Moreover, the restriction of~$\psi $ to~$N$ is the identity, and the restriction of~$\psi $ to~$H$ is trivial.
\end{lem}

In our situation, we obtain~$\psi \colon \U_{m+1} \to V_m^*$, which maps an element of~$\U_m$ to its first row, with the first entry deleted; the map~$g \mapsto g \cdot \psi (g)$ is a~$1$-cocycle.

We point out that~$\phi $ and~$\psi $ play exactly symmetrical roles, but the definition of ``$1$-cocycle'' is not correspondingly symmetrical. One could use ``homogeneous cocycles'' instead, to regain the symmetry, but the expressions would again be more complicated.

We now define $1$-cocycles for~$\U_{m+1} \times \U_{m+1}$ by ``inflating'' using the projections, namely we define $\phi_1(g_1, g_2) = \phi (g_1)$, and~$\phi_2(g_1, g_2) = g_2 \psi (g_2)$, so that~$\phi_1 \in \h^1(\U_{m+1} \times \U_{m+1}, V_m)$ and~$\phi_2 \in \h^1(\U_{m+1} \times \U_{m+1}, V_m^*)$.

From Proposition~\ref{prop-S-is-tensor-product}, it makes sense to talk about the cup-product 
\[ \phi_1 \cupp \phi_2 \in \h^2(\U_{m+1} \times \U_{m+1}, S_n) \, ,   \]
and this is the object of the next section.

\section{The cup-product} \label{sec-the-cup-product}

\begin{thm} \label{thm-coho-class-is-cup-product}
Let~$n= 2m+1$. The cohomology class of the extension 
\[ \begin{CD}
0 @>>> S_n @>>> \U_n(\f_p) @>>> \U_{m+1}(\f_p) \times \U_{m+1}(\f_p) @>>> 1
\end{CD}
  \]
  is~$\phi_1 \cupp \phi_2 \in \h^2(\U_{m+1} \times \U_{m+1}, S_n)$, where~$\phi_1$ and~$\phi_2$ were introduced in the last section.
\end{thm}

This section is entirely devoted to the (largely computational) proof. We use the notation~$G = \U_{m+1} \times \U_{m+1}$, writing~$G= G_1 \times G_2$, and we start by defining a set-theoretic section~$s \colon G \to \U_n$. Suppose 
\[ g_1= \left(\begin{array}{c|c}
  h_1 & C \\
  \hline
  0 & 1
\end{array}\right) \qquad\textnormal{and}\qquad g_2 = \left(\begin{array}{c|c}
  1 & R \\
  \hline
  0 & h_2
\end{array}\right) \]
are elements of~$\U_{m+1}$, where~$h_1, h_2 \in \U_m$, and~$C$ is a column, while~$R$ is a row. We point out that~$h_1= \pi(g_1)$, $C= \phi (g_1)$, $h_2= \pi'(g_2)$ and~$R= \psi (g_2)$, in the notation introduced previously.

Then we put 
\[ s(g_1, g_2) = \left(\begin{array}{c|c|c}
  h_1 & C & 0 \\
  \hline
  0 & 1 & R \\
  \hline
  0 & 0 & h_2
\end{array}\right) \, .   \]
We observe a useful property at once. Suppose that~$g_1', g_2', h_1', h_2', C', R'$ are similar elements, related as above, and that~$X, Y$ are arbitrary $m \times m$-matrices, then 
\[
\left(\begin{array}{c|c|c}
  h_1 & C & X \\
  \hline
  0 & 1 & R \\
  \hline
  0 & 0 & h_2
\end{array}\right) \left(\begin{array}{c|c|c}
  h_1' & C' & Y \\
  \hline
  0 & 1 & R' \\
  \hline
  0 & 0 & h'_2
\end{array}\right) = \left(\begin{array}{c|c}
  g_1 g_1' & h_1 Y +CR' + X h_2' \\
  \hline
  0 & g_2 g_2'
\end{array}\right) \, .  \tag{*}
  \]
Here we are cheating with the notation a little: on the right-hand-side, this is not really a block-matrix notation, since~$g_1 g_1'$ and~$g_2 g_2'$ are in fact overlapping. Since they are both elements of~$\U_{m+1}$, the bottom-right entry of~$g_1 g_1'$ is a~$1$, as is the top-left entry of~$g_2 g_2'$, so the overlap makes sense.

We have a bijection of sets~$S_n \times G \to \U_n$ defined by~$(m, g) \mapsto m \, s(g)$. The inverse bijection will be written~$\sigma \mapsto (\sigma_S, \sigma_G)$. If~$q \colon \U_n \to G$ is the quotient map, we have of course $\sigma_G = q(\sigma ) $, while $\sigma_S = \sigma s(q(\sigma ))^{-1}$. Here it seems more intuitive to write~$\sigma_0$ for~$s(q(\sigma ))$, because for
\[ \sigma = \left(\begin{array}{c|c|c}
  h_1 & C & X \\
  \hline
  0 & 1 & R \\
  \hline
  0 & 0 & h_2
\end{array}\right), \qquad\textnormal{we have}~ \sigma_0 = \left(\begin{array}{c|c|c}
  h_1 & C & 0 \\
  \hline
  0 & 1 & R \\
  \hline
  0 & 0 & h_2
\end{array}\right) \, .   \]
In this notation, we have~$\sigma_S = \sigma \sigma_0^{-1}$.

The multiplication of~$\U_n$, transported to~$S_n \times G$ {\em via} our bijection, is of the form 
\[ (m, g) (m', g') = (m + g \cdot m' + c(g, g'), gg') \, ,  \]
using an additive notation on~$S_n$ here. The cocycle~$c$ is what we are after, and for~$m = m' = 0$ we find, as usual, that 
\[ c(g, g') = [s(g) s(g')]_S \, ,   \]
so in the end 
\[ c(g, g') = [s(g) s(g')] \, [s(g) s(g')]_0^{-1} \, .   \]

Now we use (*) repeatedly, firstly to obtain, with~$g= (g_1, g_2)$, $g'= (g_1', g_2')$ and all other notation as above:
\[ s(g)s(g') = \left(\begin{array}{c|c}
  g_1 g_1' & CR' \\
  \hline
  0 & g_2 g_2'
\end{array}\right) \, , \qquad ~\textnormal{so}~ [s(g)s(g')]_0 = \left(\begin{array}{c|c}
  g_1 g_1' & 0 \\
  \hline
  0 & g_2 g_2'
\end{array}\right) \, .   \]
We look for~$[s(g) s(g')]_0^{-1}$ in the form 
\[ \left(\begin{array}{c|c}
  (g_1 g_1')^{-1} & Y \\
  \hline
  0 & (g_2 g_2')^{-1}
\end{array}\right) \, ; \]
indeed, from (*) again, we see that, in order to have $[s(g) s(g')]_0[s(g) s(g')]_0^{-1}= 1$ we only need to have 
\[ \pi(g_1g_1') Y + \phi (g_1 g_1')  \psi [(g_2 g_2')^{-1}]  = 0 \, , \tag{**}  \]
which has a unique solution~$Y$ since~$\pi(g_1g_1')$ is invertible. Finally, we compute the product $[s(g) s(g')][s(g) s(g')]_0^{-1}$ using (*) one last time; the result is a matrix in~$S_n$, and in the top-right corner we find 
\[ \pi(g_1g_1') Y + \phi (g_1 g_1')  \psi [(g_2 g_2')^{-1}] + CR' \pi'(g_2 g_2')^{-1} = CR' \pi'(g_2 g_2')^{-1} \, ,     \]
taking (**) into account.

So our final result is 
\[ c(g, g') = CR' \pi'(g_2 g_2')^{-1} = \phi(g_1) \psi(g'_2) \pi'(g_2 g_2')^{-1} \, ,   \]
which is also, thinking of~$S_n$ as~$V_m \otimes V_m^{*}$: 
\[ c(g, g') = \phi_1(g) \otimes (gg') \cdot \psi (g_2')= \phi_1(g) \otimes g \cdot \phi_2(g') \, ,  \]
which reads $c = \phi_1 \cupp \phi_2$, QED.

\section{The wreath extensions} \label{sec-wreath-extensions}

Using the work of the preceding sections, we proceed to describe some finite groups having~$\U_n$ as a quotient. Before we get to that however, we start by defining groups~$\U_{m+1} ^{(i)}$, for~$i= 1, 2$, having~$\U_{m+1}$ as a quotient, where we have fixed~$n= 2m+1$. Perhaps surprisingly, these two groups will be defined to be~$\f_p[\U_m] \rtimes \U_m$, the semidirect product with left multiplication of~$\U_m$ on the group algebra~$\f_p[\U_m]$ as the action. The point is that~$\U_{m+1} ^{(1)}$ and~$\U_{m+1} ^{(2)}$ will be equipped with very different maps to~$\U_{m+1}$, and we will consistently use upperscripts to keep the dichotomy in mind. Also, we will be interested in the product~$\U_{m+1} ^{(1)} \times \U_{m+1} ^{(2)}$, and again the upperscripts will be helpful to discuss the factors. More justification for the notation is given below. These groups will be called the first and second, respectively, {\em wreath extensions} of~$\U_{m+1}$, which makes sense since~$\f_p[\U_m] \rtimes \U_m$  can also be described as a wreath product~$C_p \wr \U_m$.

So we start with considering the~$\U_m$-module~$V_m$, which is cyclic, that is, generated by a single element, for example $v= {}^t (0, 0, \ldots, 0, 1)$. We shall consider the surjective map $\f_p[\U_m] \to V_m$ which takes~$1 \in \f_p[\U_m]$ to~$v$. With this in mind, we define~ $f_1 \colon \f_p[\U_m] \rtimes \U_m \to V_m \rtimes \U_m$ to be the homomorphism that maps~$1 \in \f_p[\U_m]$ to~$v$, and is the identity on~$\U_m$. Lemma~\ref{lem-Um-is-semidirect} allows us to identify~$\U_{m+1}$ with~$V_m \rtimes \U_m$, so we view~$f_1$ as a map~$\U_{m+1} ^{(1)} \to \U_{m+1}$. Thus, we think of~$\U_{m+1} ^{(1)}$ as ``$\U_{m+1}$ with its rightmost column enlarged to become all of~$\f_p[\U_m]$ rather than just~$V_m$''.

%% We shall consider the surjective map~$\f_p[\U_m] \to V_m$ which takes~$1 \in \f_p[\U_m]$ to~$v$. With this in mind, we define~$\U_{m+1} ^{(1)} := \f_p[\U_m] \rtimes \U_m$; this group maps onto~$\U_{m+1}$ {\em via} the obvious homomorphism $f_1 \colon \f_p[\U_m] \rtimes \U_m \to V_m \rtimes \U_m$.  Here we use Lemma~\ref{lem-Um-is-semidirect} to see~$\U_{m+1}$ as~$V_m \rtimes \U_m$, and~$f_1$ is then the identity on~$\U_m$, and maps~$1 \in \f_p[\U_m]$ to~$v$. We call~$\U_{m+1} ^{(1)}$ the {\em first wreath extension} of~$\U_{m+1}$, which is justified since this group can also be seen as a wreath product~$\f_p \wr \U_m$.

The second wreath extension $\U_{m+1} ^{(2)}$ will be obtained by considering~$V_m^*$, which is generated by~$v^* = (1, 0, \ldots, 0)$. This time we view~$\U_{m+1}$ as~$V_m^* \rtimes \U_m$ as in Lemma~\ref{lem-Um-is-semidirect-dual}; and we look at the  map~$f_2 \colon \f_p[\U_m] \rtimes \U_m \to V_m^* \rtimes \U_m$ which extends~$1 \mapsto v^*$. We view it as a map~$f_2 \colon \U_{m+1} ^{(2)} \to \U_{m+1}$, and we think of~$\U_{m+1} ^{(2)}$ as ``$\U_{m+1}$ with its first line enlarged to be all of~$\f_p[\U_m]$''.

We shall write~$\tilde G := \U_{m+1} ^{(1)} \times \U_{m+1} ^{(2)}$, which has a canonical map to~$G  = \U_{m+1} \times \U_{m+1}$, namely $f= (f_1, f_2)$. Using the homomorphism $\tilde G \to \U_m$ defined by $(g_1, g_2) \mapsto \pi(f_1(g_1))$, we can see~$\f_p[\U_m]$ as a~$\tilde G$-module, to be written~$\f_p[\U_m] ^{(1)}$. Using~$(g_1, g_2) \mapsto \pi'(f_2(g_2))$, we define~$\f_p[\U_m] ^{(2)}$.

The next lemma refers to the~$1$-cocycles~$\phi_1 \colon G \to  V_m$ and~$\phi_2 \colon G \to  V_m^*$ defined previously.

\begin{lem}
  There exists a~$1$-cocycle~$\tilde \phi_1 \colon \tilde G \to \f_p[\U_m] ^{(1)}$ which lifts~$\phi_1 \circ f \colon \tilde G \to V_m$. Likewise, there exists a~$1$-cocycle~$\tilde \phi_2 \colon \tilde G \to \f_p[\U_m] ^{(2)}$ which lifts~$\phi_2 \circ f$.

  In cohomological language, the map 
\[  \h^1(\tilde G, \f_p[\U_m] ^{(1)}) \longrightarrow \h^1(\tilde G, V_m) \]
induced by~$1 \mapsto v$ takes~$[\tilde \phi_1]$ to~$[f^*(\phi_1)]$, and the map
\[ \h^1(\tilde G, \f_p[\U_m] ^{(2)}) \longrightarrow \h^1(\tilde G, V_m^*)  \]
induced by~$1 \mapsto v^*$ takes~$[\tilde \phi_2]$ to~$[f^*(\phi_2)]$.
\end{lem}

\begin{proof}
Simply apply Lemma~\ref{lem-semidirect-gives-cocycle} to~$\f_p[\U_m] \rtimes \U_m$, and inflate the resulting~$1$-cocycle to~$\tilde G$ using either projection.
\end{proof}

Let us write simply~$\f_p[\U_m \times \U_m]$ for the tensor product~$\f_p[\U_m] ^{(1)} \otimes \f_p[\U_m] ^{(2)}$, as a~$\tilde G$-module. The Lemma implies that the cup-product~$\tilde \phi_1 \cupp \tilde \phi_2 \in \h^2(\tilde G, \f_p[\U_m \times \U_m])$ maps to~$f^*(\phi_1 \cupp \phi_2) \in \h^2(\tilde G, V_m \otimes V_m^*)$, under the natural map on the level of coefficients.

Now, let us consider the extension of~$\tilde G$ defined by~$f^*(\phi_1 \cupp \phi_2)$; this has the form
\[ 0 \longrightarrow S_n \longrightarrow \U_n^{(1, 2)} \longrightarrow \tilde G \longrightarrow 1 \, ,   \]
where~$\U_n^{(1,2)}$ is the pull-back of~$\U_n$ along~$f$.

Similarly, we use~$\tilde \phi_1 \cupp \tilde \phi_2$ to define the extension 
\[ 0 \longrightarrow \f_p[\U_m \times \U_m] \longrightarrow  \tilde \U_n \longrightarrow \tilde G \longrightarrow 1 \, ,   \]
which finally provides our definition of~$\tilde \U_n$. To summarize, there is a commutative diagram 
\[ \begin{CD}
  \f_p[\U_m \times \U_m] @>{1 \otimes 1 \mapsto v \otimes v^*}>> V_m \otimes V_m^* @>=>> V_m \otimes V_m^* \\
  @VVV                        @VVV                     @VVV    \\
  \tilde \U_n @>>> \U_n ^{(1,2)} @>>> \U_n \\
  @VVV  @VVV @VVV \\
  \tilde G @>=>> \tilde G @>f>> G \, , 
\end{CD}
  \]
and the square on the bottom right is a pullback diagram. All the arrows are surjective. 

The extension defining~$\tilde \U_n$ will be easy to study, for the following reason. Recall the subgroups~$N = \ker(\pi)$ and~$N'= \ker(\pi')$, in~$\U_{m+1}$; recall also that~$ N_1 \times N_2' \subset G_1 \times G_2 = \U_{m+1} \times \U_{m+1}$ is the kernel of the action of~$G$ on~$V_m \otimes V_m^*$. Perhaps more important to us is that~$G/ N_1 \times N_2' = \U_m \times \U_m$. From this it follows readily that:

\begin{lem}
Let~$\tilde N = f^{-1}( N_1 \times N_2') \subset \tilde G$. Then~$\tilde G / \tilde N = \U_m \times \U_m$, and as a result, the module~$\f_p[\U_m \times \U_m]$ is induced from the trivial module of~$\tilde N$. In symbols 
\[ \f_p[\U_m \times \U_m] = \Ind_{\tilde N}^{\tilde G}(\f_p) \, .  \qquad \square \]
\end{lem}

\begin{rmk}
By contrast, the $G$-module~$V_m \otimes V_m^*$ is not induced -- this is the chief reason for prefering to work with~$\tilde G$ and the corresponding extension. Indeed, the starting point of our work was to find a suitable extension of~$\U_n(\f_p)$, fitting in an exact sequence, whose kernel is an induced module. This allows the use of Shapiro's Lemma, as in the next Theorem. In~\cite{previous}, Shapiro's Lemma played a pivotal role.  
\end{rmk}

In the following statement, we use the decomposition~$\tilde N = \tilde N_1 \times \tilde N_2'$ where~$\tilde N_1 =f_1^{-1}(N_1) \cong \f_p[\U_m] \subset \U ^{(1)}_{m+1}$ and~$\tilde N_2' = f_2^{-1}(N_2') \cong \f_p[\U_m] \subset \U ^{(2)}_{m+1}$.

\begin{thm} \label{thm-maps-into-Un-tilde}
Let~$\Gamma $ be a profinite group, and let~$\gamma \colon \Gamma \to \tilde G$ be a continuous homomorphism. Let~$\Lambda = \gamma^{-1}(\tilde N)$, and let~$\lambda \colon \Lambda \to \tilde N$ be the restriction of~$\gamma $. Then the following statements are equivalent. \begin{enumerate}
\item $\gamma $ can be lifted to a continuous homomorphism $\Gamma \to \tilde \U_n$.
  \item The subspaces~$\lambda^*(\h^1(\tilde N_1, \f_p))$ and~$\lambda^*(\h^1(\tilde N_2', \f_p))$ are orthogonal for the cup-product. In more details, for any~$x \in \h^1(\tilde N_1, \f_p)$ and any~$y \in \h^1(\tilde N_2', \f_p)$, which we may see both as elements of~$\h^1(\tilde N, \f_p)$, we have 
\[ \lambda^*(x) \cupp \lambda^*(y) = 0 \in \h^2(\Lambda, \f_p) \, .   \]
\end{enumerate}

%%  Moreover, if we suppose further that the composition 
%% %
%% \[ \begin{CD}
%% \Gamma @>\gamma >> \tilde G @>f>> G @>{(\pi, \pi')}>> \U_m \times \U_m
%% \end{CD}
%%   \]
%%   %
%%   is surjective -- what we will call the {\em non-degenerate case} -- then the properties are also equivalent to

%% \begin{enumerate}
%% \item[(3)] The single cup-product~$\lambda^*(x_0) \cupp \lambda^*(y_0) $ vanishes, where~$x_0, y_0 \in \h^1(\tilde N, \f_p)$ are distinguished elements, to be given in the proof.
%% \end{enumerate}
  
\end{thm}

\begin{proof}
  Let use write~$\omega = \tilde \phi_1 \cupp \tilde \phi_2$, the cohomology class defining~$\tilde \U_n$. Then it is clear that (1) is equivalent to~$\gamma^*(\omega ) = 0 \in \h^2(\Gamma , \f_p[\U_m \times \U_m])$.

By the last Lemma, the module~$\f_p[\U_m \times \U_m]$ is free as a~$\tilde G / \tilde N$-module. Therefore, it is also free as a module over the subgroup~$\Gamma / \Lambda \subset \tilde G / \tilde N$. As a result, it is of the form~$\Ind_\Lambda^\Gamma (\f_p^d)$ for some~$d \ge 1$. By Shapiro's Lemma, the restriction map 
\[ \h^2(\Gamma , \f_p[\U_m \times \U_m]) \longrightarrow \h^2(\Lambda, \f_p[\U_m \times \U_m])  \]
is injective. Condition (1) is thus equivalent to~$\lambda^*( \omega |_{\tilde N}) = 0$, where $ \omega |_{\tilde N}$ is the restriction of~$\omega $ to~$\tilde N$.

Now, after we restrict to~$\tilde N$, things become surprisingly simple. All the modules considered are trivial. In fact, let~$E= \f_p^k$ where~$k = |\U_m|$. Then we can identify~$E$ with each of~$\tilde N_1$, $\tilde N_2'$, and $\f_p[\U_m] ^{(i)}$ for~$i= 1, 2$ (that is, in each case we have a canonical basis). The cocycle~$\tilde \phi_i$, once restricted, is now the identity of~$E$, for~$i= 1, 2$ (this is part of Lemma~\ref{lem-semidirect-gives-cocycle}). The class $ \omega |_{\tilde N}$ is merely $p_1 \cupp p_2 \in \h^2(E \times E, E \otimes E)$, where~$p_1$ and~$p_2$ are the first and second projection, respectively, from~$E \times E$ to~$E$. If~$\varepsilon_1, \ldots, \varepsilon_k$ is any basis of~$E^* = \h^1(E, \f_p)$, then we can also write (with a mild abuse of notation)
\[  \omega |_{\tilde N} = \sum_{i,j} \varepsilon_i \cupp \varepsilon_j \in \bigoplus_{i,j} \h^2(E \times E, \f_p) \, .   \]
The cohomology group~$\h^2(\Lambda , E \otimes E)$ likewise splits as a direct sum of copies of~$\h^2(\Lambda, \f_p)$, and it is now clear that~$\lambda^*( \omega |_{\tilde N})=0$ is equivalent to (2).
\end{proof}

%
%% To improve this to (3) in the non-degenerate case, let us remember the~$\U_m$-action on~$E = \f_p[\U_m]$, and let us choose the basis formed by all~$g \cdot v_0$, where~$v_0$ corresponds to~$1 \in \f_p[\U_m]$ in our identification, where~$g$ runs through the elements of~$\U_m$. Let~$\alpha \in E^*$ be such that~$\alpha (v_0)= 1$ and~$\alpha (g\cdot v_0) = 0$ for~$g \ne 1$; a basis of~$E^*$ is then given by the various~$g \cdot \alpha $. Now~$x_0$ and~$y_0$ are the two elements of~$(E \times E)^*$ obtained from~$\alpha $ using either projection to~$E$. Condition (2) is now 
%% %
%% \[ \lambda^*(g \cdot x_0) \cupp \lambda^*(h \cdot y_0) = 0 \tag{*}  \]
%% %
%% for all~$g, h \in \U_m$.

%% The class~$ \omega |_{\tilde N}$ is obviously stable under conjugation by elements of~$\Gamma $. It follows that, if $\lambda^*(x_0) \cupp \lambda^*(y_0) = 0$, then (*) holds for all~$(g, h) \in \U_m \times \U_m$ in the image of the composition given in the Proposition. If this composition is surjective, we see that all equations follow from one, as promised.

To conclude the group-theoretic considerations, here is a more concrete description of the group~$\tilde \U_5(\f_p)$.

\begin{prop} \label{prop-explicit-U5}
The subgroup of~$\U_{2p+1}(\f_p)$ generated by the four elements 
\[ \sigma_1' := \left(\begin{array}{rr}
  J & 0 \\
  0 & 1
\end{array}\right) \, , \qquad \sigma_2':=\sigma_p \, , \qquad \sigma_3':=\sigma_{p+1} \, , \qquad \sigma_4' := \left(\begin{array}{rr}
  1 & 0 \\
  0 & J
\end{array}\right) \, , 
\]
where~$J$ is a Jordan block of size~$p$, is isomorphic to~$\tilde \U_5(\f_p)$.
\end{prop}

For~$p=3$ for example, the generators are 
\[ \sigma_1' = \left(\begin{array}{rrrrrrr}
  1 & 1 & 0 & 0 & 0 & 0 & 0 \\
  0 & 1 & 1 & 0 & 0 & 0 & 0 \\
  0 & 0 & 1 & 0 & 0 & 0 & 0 \\
  0 & 0 & 0 & 1 & 0 & 0 & 0 \\
  0 & 0 & 0 & 0 & 1 & 0 & 0 \\
  0 & 0 & 0 & 0 & 0 & 1 & 0 \\
  0 & 0 & 0 & 0 & 0 & 0 & 1 
\end{array}\right) \, , \qquad \sigma_2' = \left(\begin{array}{rrrrrrr}
  1 & 0 & 0 & 0 & 0 & 0 & 0 \\
  0 & 1 & 0 & 0 & 0 & 0 & 0 \\
  0 & 0 & 1 & 1 & 0 & 0 & 0 \\
  0 & 0 & 0 & 1 & 0 & 0 & 0 \\
  0 & 0 & 0 & 0 & 1 & 0 & 0 \\
  0 & 0 & 0 & 0 & 0 & 1 & 0 \\
  0 & 0 & 0 & 0 & 0 & 0 & 1 
\end{array}\right) \, ,  \]
\[ \sigma_3' = \left(\begin{array}{rrrrrrr}
  1 & 0 & 0 & 0 & 0 & 0 & 0 \\
  0 & 1 & 0 & 0 & 0 & 0 & 0 \\
  0 & 0 & 1 & 0 & 0 & 0 & 0 \\
  0 & 0 & 0 & 1 & 1 & 0 & 0 \\
  0 & 0 & 0 & 0 & 1 & 0 & 0 \\
  0 & 0 & 0 & 0 & 0 & 1 & 0 \\
  0 & 0 & 0 & 0 & 0 & 0 & 1 
\end{array}\right) \, , \qquad \sigma_4' = \left(\begin{array}{rrrrrrr}
  1 & 0 & 0 & 0 & 0 & 0 & 0 \\
  0 & 1 & 0 & 0 & 0 & 0 & 0 \\
  0 & 0 & 1 & 0 & 0 & 0 & 0 \\
  0 & 0 & 0 & 1 & 0 & 0 & 0 \\
  0 & 0 & 0 & 0 & 1 & 1 & 0 \\
  0 & 0 & 0 & 0 & 0 & 1 & 1 \\
  0 & 0 & 0 & 0 & 0 & 0 & 1 
\end{array}\right) \, . 
\]
This allows you to enter the matrices in a computer algebra system, such as GAP, and immediately obtain information about~$\tilde \U_{5}(\f_p)$ for a concrete value of~$p$. For a random example, with~$p=3$, GAP tells us that the successive subgroups in the lower central series have order~$3^{17}$, $3^{13}$, $3^{10}$, $3^6$, $3^3$, $3$ and~$1$ respectively.

On the other hand, for~$p=2$, this provides the most convincing proof that~$\tilde \U_5(\f_2) = \U_5(\f_2)$. 

\begin{proof}
  Let~$\sigma= \sigma_1 $ be the canonical generator for the group~$\U_2$, which is cyclic of order~$p$. The action of~$\sigma $ on the group algebra~$\f_p[\U_2]$ is given by a Jordan block such as~$J$, if one works in the basis~$u^{p-1}, u^{p-2}, \ldots, u, 1$, where~$u = \sigma - 1$. This provides a map $\U_2 \to \U_p$ mapping~$\sigma $ to~$J$. We use the specified basis to identity~$\f_p[\U_2]$ with~$\f_p^p = V_p$, the natural module for~$\U_p$, and we extend our map to a homomorphism $\U_3 ^{(1)} = V_p \rtimes \U_2 \to V_p \rtimes \U_p = \U_{p+1}$, which is the identity on~$V_p$. It is visibly injective, and its image is generated by the elements obtained from~$\sigma_1'$ and~$\sigma_2'$ as given above, when we keep only their~$p+1$ first rows and columns.

  The action of~$\sigma $ on the dual~$V_p^*$ can likewise be put in Jordan canonical form, and a similar discussion provides us with a monomorphism~$\U_3 ^{(2)} = V_p^* \rtimes \U_2 \to V_p^* \rtimes \U_p = \U_{p+1}$, which is the identity on~$V_p^*$. Generators for its image can be obtained by truncating~$\sigma_3'$ and~$\sigma_4'$.

  Combining the two, we end up with an injective homomorphism 
\[\iota \colon  \tilde G = \U_3 ^{(1)} \times \U_3 ^{(2)} \longrightarrow \U_{p+1} \times \U_{p+1} \, .  \]
Of course, we can see~$\U_{p+1} \times \U_{p+1}$ as~$\U_{2p+1} / S_{2p+1}$. When we do that, the image of~$\iota $ is generated by the images of~$\sigma_i'$ for~$i= 1, 2, 3, 4$, under the natural projection.

Let~$\alpha \in \h^2(\U_{p+1} \times \U_{p+1}, S_{2p+1})$ be the class of the extension 
\[ 0 \longrightarrow S_{2p+1} \longrightarrow \U_{2p+1} \longrightarrow \U_{p+1} \times \U_{p+1} \longrightarrow 1 \, .   \]
Our explicit description of~$\alpha $ makes it clear that~$\iota^*(\alpha )$ is the class of 
\[ 0 \longrightarrow S_{2p+1} \longrightarrow \tilde \U_5 \longrightarrow \tilde G \longrightarrow 1 \, .   \]
(Note that~$S_{2p+1} = V_p \otimes V_p^*$.) In particular, this describes $\tilde \U_5$ as a fibre product
\[ \tilde \U_5 = \U_{2p+1} \times_{\U_{p+1} \times \U_{p+1}} \tilde G \, .   \]
More precisely however, we see that~$\iota $ lifts to an embedding~$\kappa \colon \tilde \U_5 \to \U_{2p+1}$, whose image is the full pre-image of $\operatorname{Im}(\iota )$.

To obtain the statement of the Proposition, it is therefore enough to show that the four elements~$\sigma_i'$, for~$i= 1, 2, 3, 4$, generate a group containing~$S_{2p+1}$. For this, we simply compute the commutator~$[\sigma_2', \sigma_3']$, which is the identity matrix with an extra~$1$ in position~$(p, p+2)$; for~$p=3$ this is 
\[ [\sigma_2', \sigma_3'] = \left(\begin{array}{rrrrrrr}
  1 & 0 & 0 & 0 & 0 & 0 & 0 \\
  0 & 1 & 0 & 0 & 0 & 0 & 0 \\
  0 & 0 & 1 & 0 & 1 & 0 & 0 \\
  0 & 0 & 0 & 1 & 0 & 0 & 0 \\
  0 & 0 & 0 & 0 & 1 & 0 & 0 \\
  0 & 0 & 0 & 0 & 0 & 1 & 0 \\
  0 & 0 & 0 & 0 & 0 & 0 & 1 
\end{array}\right) \, .   \]
Under the isomorphism~$S_{2p+1} \cong V_p \otimes V_p^* \cong \f_p[\U_2] \otimes \f_p[\U_2]$, this is the element $1 \otimes 1$. The result follows.
\end{proof}

\begin{rmk}
  In the sequel, the maps~$s_i$ defined on~$\U_n$ will be seen as maps on~$\U_{m+1} ^{(i)}$ or on~$\tilde \U_n$, when necessary. What is meant is to use the quotient~$\U_{m+1} ^{(i)} \to \U_{m+1}$, or~$\tilde \U_n \to \U_n$, and then compose with the maps written~$s_i$ in \S\ref{sec-prelim}. For example, on $\U_3 ^{(1)} = \f_p[\U_2] \rtimes \U_2$, the kernel of~$s_1$ is~$\f_p[\U_2]$.

  Also, the group~$\tilde \U_5$ having four distinguished elements~$s_1, s_2, s_3, s_4 \in \h^1(\tilde \U_5, \f_p)$, we can now speak of the vanishing of cohomology classes ``in the sense of~$\tilde \U_5$'', as promised in the Introduction, and explained in Remark~\ref{rmk-G-massey-products}. Explicitly, let~$\Gamma $ be a profinite group with classes~$\chi_1, \chi_2, \chi_3, \chi_4 \in \h^1(\Gamma , \f_p)$. We say that the Massey product of these classes vanishes in the sense of~$\tilde \U_5$ when there exists a continuous~$\phi \colon \Gamma \to \tilde \U_5$ such that~$s_i \circ \phi = \chi_i$ for~$i= 1, 2, 3, 4$. We also say that~$\phi $ is compatible with the $\chi_i$'s. 
\end{rmk}

We state a corollary with this vocabulary:

\begin{coro} \label{coro-repeated-massey}
Suppose~$\Gamma $ is a profinite group, and that the Massey product of the classes~$\chi_i \in \h^1(\Gamma , \f_p)$, for~$1 \le i \le 4$, vanishes in the sense of~$\tilde \U_5$. Then the Massey product 
\[ \langle \chi_1, \chi_1, \ldots, \chi_1, \chi_2, \chi_3, \chi_4, \chi_4, \ldots, \chi_4  \rangle  \]
also vanishes, where~$\chi_1$ and~$\chi_4$ are repeated less than~$p$ times.
\end{coro}

\begin{proof}
  Compose~$\phi \colon \Gamma \to \tilde \U_5$ with the embedding $\tilde \U_5 \to \U_{2p+1}$ constructed in the proof of the last Proposition. This gives the vanishing of
\[ \langle \chi_1, \chi_1, \ldots, \chi_1, \chi_2, \chi_3, \chi_4, \chi_4, \ldots, \chi_4  \rangle \, ,  \]
where~$\chi_1$ and~$\chi_4$ are repeated exactly~$p-1$ times. Then use the maps~$\pi \colon \U_n \to \U_{n-1}$ or~$\pi' \colon \U_n \to \U_{n-1}$ as often as you want.
\end{proof}

\section{Translation into Galois theory} \label{sec-translation}

We proceed to study some Galois-theoretic situations involving the groups~$\U_n$, $\U_{m+1} ^{(i)}$, or $\tilde \U_n$. All fields  will be assumed to contain a primitive~$p$-th root of unity~$\dz$, where~$p$ is our usual fixed prime. 

We start with a general situation where we can produce extensions with Galois group isomorphic to~$\f_p[\U_n] \rtimes \U_n$. The Proposition below will be mostly used for~$n=2$, but the general case is interesting: it is a rare example (perhaps the only one at the time of writing) of a situation when one can show the vanishing of an~$n$-fold Massey product, given the vanishing of an~$(n-1)$-product. Also, the formula showing up in (3) below does not seem to have appeared elsewhere.

We use freely the language introduced in~\S\ref{sec-prelim}.

\begin{prop} \label{prop-from-Un-to-Un+1}
  Let~$K/F$ be Galois with~$\Gal(K/F) = \U_n(\f_p)$, and assume that $K/F$ is compatible with~$a_1, \ldots, a_{n-1} \in F^\times$. In particular, $\langle a_1, \ldots, a_{n-1} \rangle$ vanishes. Assume that~$\dz \in F$.

  Suppose that~$a_n \in F^\times$ is a norm from~$K$, that is~$a_n = N_{K/F}(w)$ for some~$w \in K$. Finally, assume that~$a_1, \ldots, a_n$ are linearly independent in~$F^\times / F^{\times p}$.

  \begin{enumerate}
\item   Let 
\[ L = K[\sq{g(w)} : g \in \Gal(K/F)] \, .   \]
Then~$L/F$ is Galois with~$\Gal(L/F) \cong \f_p[\U_n] \rtimes \U_n$.

\item Let~$\sigma_1, \ldots, \sigma_{n-1}$ be the usual generators of the group~$\U_n(\f_p)$. Suppose that choices of~$p$-th roots~$\sq{g(w)} \in L$ have been made, for each~$g \in \Gal(K/F)$. Then there exist lifts $\tilde \sigma_1$, $\ldots$, $\tilde \sigma_{n-1} \in \Gal(L/F)$ satisfying~$\tilde \sigma_i(\sq{g(w)}) = \sq{\sigma_ig (w)}$ for all~$g \in \Gal(K/F)$. 

\item Let the extension~$M/F$ be the Galois closure of~$K[\sq{v}]/F$, where 
\[ v = \prod_{g \in \Gal(K/F)} g(w)^{\lambda_g} \, ,   \]
and~$\lambda_g \in \f_p$ is the top-right coefficient of the matrix~$g \in \U_n= \Gal(K/F)$. Then $\Gal(M/F) \cong \U_{n+1}$. In fact the homomorphism $\Gal(L/F) \to \Gal(M/F)$ can be identified with $f_1 : \f_p[\U_n] \rtimes \U_n \to \U_{n+1}$ introduced in the previous section.

\item The extension~$M/F$ is compatible with~$a_1,\ldots, a_n$, and in particular, the Massey product  $\langle a_1, \ldots, a_n \rangle$ vanishes.

%% \item The Frattini quotient of~$M/F$ is~$F[\sq{ a_1}, \ldots, \sq{ a_n}]$, and in particular, the Massey product  $\langle a_1, \ldots, a_n \rangle$ vanishes.

\end{enumerate}

\end{prop}

\begin{proof}
In this proof we will treat the isomorphism~$\Gal(K/F) = \U_n$ as an equality (or an identification, if you will), and no particular notation will be used. Also, we will use {\em equivariant Kummer theory} repeatly. Briefly, recall that when~$\dz \in K$, the extensions~$E/K$ which are Galois with~$\Gal(E/F)$ a~$p$-elementary abelian group (and with~$E$ contained in a fixed algebraic closure of~$F$) are in bijection with the finite subgroups of~$\lin K$; those~$E$ for which~$E/F$ is Galois are in bijection with the~$\f_p[\Gal(K/F)]$-submodules of~$\lin K$. The bijection is given by~$A \mapsto K[\sq a : a \in A]$. Finally, if~$E$ corresponds to~$A$, then the subgroup~$\Gal(E/K)$ of~$\Gal(E/F)$, with the conjugation action, is isomorphic to the dual of~$A$, , as is expressed by the non-degenerate Kummer pairing 
\[ \Gal(E/F) \times A \longrightarrow \mu_n(K)  \]
defined by~$\langle \sigma , a  \rangle = \frac{\sigma (\sq a)} {\sq a}$. Finally, suppose~$E$ corresponds to~$A$, let~$V= \Gal(E/K)$, and let~$V' \subset V$ be a subgroup, defining the field~$E'$ such that~$\Gal(E/E') = V'$. To find the subgroup~$A'$ of $\lin K$ corresponding to~$E'$, one simply looks at the homomorphism~$V \to V/V'$ and its transpose~$(V/V')^* \to V^* = A$; its image is~$A'$. (This last fact follows from the other ones.)
  
(1) Equivariant Kummer theory tells us that~$L/F$ is Galois, with~$\Gal(L/K) \cong W^*$, where~$W$ is the $\Gal(K/F)$-module generated by~$w$ in~$K^\times / K^{\times p}$. There is an exact sequence 
\[ 0 \longrightarrow W^* \longrightarrow \Gal(L/F) \longrightarrow \Gal(K/F) \longrightarrow 1 \, .   \]
Lemma~\ref{lem-usual-lemma-rep-theory} implies that~$W$, and so also~$W^*$, is free (here we use that~$a_n$ is not a $p$-th power in~$K$). As a result of Shapiro's Lemma, we see that $\H^2(\Gal(K/F), W^*) = 0$, and  we conclude that~$\Gal(L/F)$ is a semi-direct product, as (1) claims. This is not yet explicit, but statement (2), to which we turn, makes up for this.

  Let~$\tau_i \in \Gal(L/F)$ be a lift of~$\sigma_i$. The number of such lifts is~$|\Gal(L/K)| = p^{|G|}$, where~$G= \Gal(K/F)$. For~$g \in \Gal(K/F)$, pick a favourite~$p$-th root~$\sq{g(w)} \in L$. It follows that $\tau_i(\sq{g(w)}) = \dz ^{i_g} \sq{\sigma_i g(w)}$, where~$\dz $ is our primitive root, and~$i_g$ is an integer with~$0 \le i_g < p$.

Now, on the one hand, $\tau_i$ is determined by the integers~$i_g$, and there are~$p^{|G|}$ choices for these. However, we have pointed out that there are actually $p^{|G|}$ lifts, so all choices indeed occur. In particular, we can pick a lift~$\tau_i$ with~$i_g= 0$ for all~$g$. This proves (2).

Consider now~$G'$, the subgroup of~$\Gal(L/F)$ generated by the elements~$\tilde \sigma_i$. Then (2) shows that, for~$\tau \in G'$ and~$g \in G$, the value of~$\tau (\sq{g(w)})$ only depends on the restriction of~$\tau $ to~$K$. So if~$\tau \in G' \cap \Gal(L/K)$, then~$\tau $ fixes all the elements~$\sq{g (w)}$,  and so~$\tau $ is the identity. It follows that~$G'$ maps isomorphically onto~$G$ {\em via} the restriction, so we see more explicitly that~$\Gal(L/F)$ is a semi-direct product. It is also useful here to remark that each $\tilde \sigma_i$ fixes~$\sq{a_n}$, since $a_n$ is the product of all the elements~$g(w)$.

As a preparation for (3), we introduce more notation. Let~$\phi \in \Gal(L/K) = W^*$ which, via the Kummer pairing, has $\langle \phi, w \rangle = 1$ and $\langle \phi, g(w) \rangle = 0$ for~$g \ne 1$, or in other words, we have $\phi (\sq{w}) = \dz \sq{w}$ and~$\phi (\sq{g(w)}) = \sq{g(w)}$ for~$g \ne 1$. The map~$\f_p[\U_n] \longrightarrow W^*$ mapping~$1$ to~$\phi $ is an isomorphism.

Consider then the map of~$\U_n$-modules~$\theta \colon W^* \longrightarrow \f_p^n$ mapping~$\phi $ to~${}^t(0, 0, \ldots, 0, 1)$ (the action on~$\f_p^n$ is the natural one, in fact~$\f_p^n$ is the ``natural module~$V_n$'' already introduced). Call~$V$ the kernel, and consider the intermediate field $K \subset M \subset L$ corresponding to~$V$ (which is a subgroup of~$\Gal(L/K)$); that is~$\Gal(L/M) = V$, and 
\[ \Gal(M/F) = \Gal(L/F) / V = (W^* \rtimes \U_n) / V = (W^*/V) \rtimes \U_n = \f_p^n\rtimes \U_n = \U_{n+1} \, .   \]

We can now prove (3). By Kummer theory, the field~$M$ as just defined is $M= K[A]$ where~$A$ is the image of~$\theta^* \colon (\f_p^n)^* \to W^{**} = W$. If~$e_1, e_2, \ldots $ is the canonical basis of~$\f_p^n$, then $\theta (\phi ) = e_n$, while $(\f_p^n)^*$ is generated by~$e_1^*$ as a~$\U_n$-module, so~$A$ is the module spanned by~$\theta^*(e_1^*) = e_1^* \circ \theta $. However, for~$g \in G$ we have 
\[ e_1^* \circ \theta ( g\cdot \phi )  = e_1^*( g \cdot e_n) \, .   \]
Now~$g \cdot e_n$ is the last column of~$g$, and~$e_1^*(g \cdot e_n)$ is the coefficient in the top-right corner of~$g$, written~$\lambda_g$ in the Proposition. It follows readily that~$\theta^*(e_1^*)$, which is determined by its values on the various~$g \cdot \phi $ for~$g \in G$, is in fact given by evaluation at 
\[ v= \sum_g \lambda_g \, g \cdot w \, ,   \]
as proposed. We have (3), and we turn to (4).

The elements~$\tilde \sigma_i \in \Gal(L/F)$ defined above, for~$1 \le i < n$, will be seen as elements of~$\Gal(M/F)$ from now on. By construction, they correspond to the elements~$\sigma_i \in \U_n$ when~$\U_n$ is seen as a subgroup of~$\U_{n+1} = \f_p^n \rtimes \U_n = \Gal(M/F)$ (by inserting a column of the identity matrix on the right). On the other hand, we have the element~$\phi \in \Gal(L/K) \subset \Gal(L/F)$, which we also see now as an element of~$\Gal(M/F)$. It corresponds to the element normally called~$\sigma_n \in \U_{n+1}$.

The group~$\U_{n+1}$ is endowed with the characters~$s_i \colon \U_{n+1} \to \f_p$ that we know. We claim that, when~$s_i$ is viewed as a character of the absolute Galois group~$G_F$, we have~$s_i =\chi_{a_i}$, the character associated to~$a_i$ via the isomorphism $F^\times / F^{\times p} \cong \Hom(G_F, \f_p)$, for~$1 \le i \le n$. This implies first that~$\sq{a_n} \in M$, rather than just~$\sq{a_n} \in L$, and the rest of (4) follows.

To prove the claim, we first note that for~$1 \le i < n$, we have ~$s_i = \chi_{a_i}$ by assumption, and the point is to extend to~$i=n$. Certainly~$\sq{a_n} \in L$, so 
\[ \chi_{a_n} = c_1 s_1 + \cdots + c_n s_n \, ,    \]
where the coefficients~$c_i$ are in~$\f_p$. This follows since~$\Gal(L/F)$ is generated by~$n$ elements, so its Frattini quotient is~$\f_p^n$, and~$\Hom(\Gal(L/F), \f_p)$ is generated by~$s_1,\ldots , s_n$.

By (2), we have~$\chi_{a_n}(\tilde \sigma_i)=0$ for~$1 \le i < n$. It follows that~$c_i = 0$ for these~$i$'s, so~$\chi_{a_n} = c_n s_n$. On the other hand, by evaluating at~$\phi = \sigma_n$, we find~$1 = c_n$.
\end{proof}

One can prove, conversely, that whenever a Galois extension~$M/F$ with $\Gal(M/F) \cong \U_{n+1}$ can be embedded into an~$\f_p[\U_n] \rtimes \U_n$-extension, with all the notation as above, then~$a_n$ is a norm from~$K$, the field corresponding to the quotient~$\U_n$. We will only prove this when~$n=3$; but we do this in all generality, including degenerate cases, and we provide a lot of details.

The next Proposition gives an overview of the situation, and will be supplemented below.

\begin{prop} \label{prop-hilbert-symbol}
  Let~$F$ be a field with~$\dz \in F$, and let~$a, b \in F^\times$. The following assertions are equivalent.

  \begin{enumerate}
\item $(a, b)_F=0$. 
\item There exists~$B \in F[\sq a]$ such that~$\N_{F[\sq a]/F}(B) = b$.
\item There exists a continuous homomorphism~$\gamma \colon G_F \to \U_3(\f_p)$ such that~$s_1 \circ \gamma = \chi_a$ and~$s_2 \circ \gamma = \chi_b$.  
\item There exists a continuous homomorphism~$\gamma \colon G_F \to \U_3 ^{(1)} = \f_p[\U_2] \rtimes \U_2$ such that~$s_1 \circ \gamma = \chi_a$ and~$s_2 \circ \gamma = \chi_b$. 

  \end{enumerate}

\end{prop}

%% \note{The letter $\gamma $ above used to be $\phi $. It makes more sense now, because of thm~\ref{thm-maps-into-Un-tilde}. I hope I have changed all the $\phi $'s into $\gamma $'s... }

\begin{proof}
The equivalence of (1) and (2) is classical, see~\cite[Chap.\ XIV, \S2, Prop.~4]{corps-locaux}. The equivalence of (1) and (3) is obtained by noting that the cohomology class of the extension 
\[ 0 \longrightarrow \f_p \longrightarrow \U_3 \longrightarrow C_p \times C_p \longrightarrow 1  \]
is the cup-product~$s_1 \cupp s_2$, while~$(a,b) = \chi_a \cupp \chi_b$. The implication (4) $\implies$ (3) is trivial, as~$\U_3$ is a quotient of~$\U_3 ^{(1)}$. However, Proposition~\ref{prop-4-implies-3} below gives more precise information. 

It remains to prove that (1), (2), and/or (3) $\implies$ (4), and this is done in Proposition~\ref{prop-2-implies-4} below, with extra details.
\end{proof}

\begin{prop} \label{prop-4-implies-3}
  With notation as in Proposition~\ref{prop-hilbert-symbol}, suppose~$\gamma $ as in (4) is given. Let $K= F[\sq a]$, so that~$G_K = \gamma^{-1}(\f_p[\U_2])$. Then we may pick an element~$B \in K^\times$ such that~$\N_{K/F}(B)= b f^p$ for some~$f \in F^\times$, and such that the character~$\chi_B \in \h^1(K, \f_p)$ is in the image of 
\[ \gamma ^* \colon \h^1(\f_p[\U_2], \f_p ) \longrightarrow \h^1(K , \f_p) \, .   \]
\end{prop}

\begin{proof}
  % this was full of mistakes:
  %% Let us dispose of some easy cases first. If~$a$ is a~$p$-th power in~$F$, then~$K=F$, we take~$B= b$, so~$\chi_B  = \chi_b = s_2 \circ \gamma = \gamma^*(s_2)$. Now, suppose that~$a$ is not a~$p$-th power, and suppose that~$a$ and~$b$ are linearly dependent mod~$p$-th powers, that is~$b = a^i f^p$ for some integer~$i$ and some~$f \in F^\times$. We have dealt with~$p=2$ in [PREVIOUS PAPER], so assume that~$p$ is odd, and put~$B= (\sq a)^i$. We have then~$\N_{K/F}(B)= a^i = b f^{-p}$. The property required of~$B$ is again obvious (in fact~$\chi _B = i \chi_{\sq a} = i \gamma^*(s_1)$).

  Let us dispose of an easy case first: if~$a$ is a~$p$-th power in~$F$, then~$K=F$, we take~$B= b$, so~$\chi_B  = \chi_b = s_2 \circ \gamma = \gamma^*(s_2)$. Now suppose that~$a$ is not a~$p$-th power.
  
  %% So we come to the interesting case when~$a$ and~$b$ are linearly independent mod~$p$-th powers. It follows that~$\gamma \colon G_F \to \U_3 ^{(1)}$ is surjective (since it is surjective modulo Frattini).
  
  The cleanest proof will be obtained by a careful examination of the corestriction map in group cohomology. Let~$\Gamma $ be a profinite group, and let~$N$ be an open subgroup; the corestriction 
\[ \cores = \cores_{N, \Gamma } \colon \h^1(N, \f_p) = \Hom(N, \f_p) \longrightarrow \h^1(\Gamma , \f_p) = \Hom(\Gamma, \f_p)  \]
has the following explicit description. Let~$\alpha \in \Hom(N, \f_p)$, then~$\cores(\alpha )$ is $\alpha ' \colon \Gamma \to \f_p$ given by 
\[ \alpha '(g) = \sum_{t \in T} \alpha (n_{tg}) \, .   \tag{*} \]
Here we have written~$\Gamma $ as the disjoint union of cosets~$N t$ for~$t \in T$; any element~$g \in \Gamma $ is then written~$g = n_g t_g$ with~$n_g \in N$ and~$t_g \in T$. When~$N$ is normal in~$\Gamma $, we have for~$n \in N$ the formula~$\alpha' (n) = \sum_t \alpha (t n t^{-1})$. If~$N$ is assumed to be abelian as well, for purposes of intuition we note that we have~$\alpha '(n) = \alpha (\N(n))$ where
\[ \N(n) = \sum_t t \cdot  n   \]
(additive notation on~$N$, the action is conjugation). Another interesting particular case is when~$T = \{ 1, \sigma, \ldots, \sigma^{p-1} \}$ for some element~$\sigma \in \Gamma $; then~$\alpha ' (\sigma ) = \alpha (\sigma^p)$.

As a warm-up, apply this to~$\Gamma = \U_3 ^{(1)} = \f_p[\U_2] \rtimes \U_2$ and~$N = \f_p[\U_2]$ (which is abelian and normal); take~$T = \U_2 = \{ 1, \sigma, \ldots, \sigma^{p-1} \}$ for some element~$\sigma $ satisfying~$\sigma^p = 1$. If~$\alpha \colon N \to \f_p$ is any linear form taking the value~$1$ on 
\[ \N(1)  = 1 + \sigma \cdot 1+ \sigma^2 \cdot 1 + \cdots + \sigma^{p-1} \cdot 1 \in N \, ,   \]
then it follows that~$\alpha ' = \cores_{N, \Gamma }(\alpha )$ is none other than~$s_2$. For this, it may useful to notice that~$s_2 \colon \U_3 ^{(1)} \to \f_p$ can be described as the composition 
\[ \U_3 ^{(1)} \longrightarrow  \U_3 ^{(1)} / \rad(N) = N / \rad(N) \times \U_2 \longrightarrow  N / \rad(N) \cong  \f_p \, .   \]
Here~$\rad(N)$ is the image of~$\sigma - 1$ on~$N$ (which is actually the radical of the~$C_p$-module~$N$, as in representation theory). Alternatively, an explicit formula for~$s_2$ is given in the next proof. Then we check that~$s_2$ and~$\alpha '$ take the same value on~$1 \in N$ (namely~$1$), and on~$\sigma $ (namely~$0$), proving that~$\alpha' = s_2$.

It is only marginally more complicated to deal with~$\Gamma  = G_F$ and~$N = G_K$ (which is normal but not abelian in general), and~$T = \{ 1, \tilde \sigma , \ldots, \tilde \sigma^{p-1} \}$ for some element~$\tilde \sigma \in G_F$ with~$\gamma (\tilde \sigma) = \sigma \in \U_3 ^{(1)}$ (such a~$\tilde \sigma $ exists by our assumption on~$a$). Here~$\tilde \sigma^{p} \in \ker(\gamma )$. Let~$\tilde \alpha = \alpha \circ \gamma $ and~$\tilde \alpha ' = \cores_{G_K, G_F}(\tilde \alpha )$, where~$\alpha $ was just discussed.

The expression (*) shows that~$\tilde \alpha' $ vanishes on~$\ker(\gamma )$ (a normal subgroup contained in~$N$), as does~$\tilde \alpha $. That~$\tilde \alpha ' = s_2 \circ \gamma$ is now easily deduced from the previous case.

The rest of the argument is just a translation into Galois-theoretic language, letting~$B \in \lin K$ be the element such that~$\chi_B = \tilde \alpha = \gamma ^*(\alpha ) \in \h^1(K, \f_p)$; its corestriction is~$\tilde \alpha ' = \chi_b \in \h^1(F, \f_p)$.
\end{proof}

This Proposition establishes that (4) $\implies$ (2) in Proposition~\ref{prop-hilbert-symbol}, for~$\N_{K/F}(Bf^{-1}) = b$ (note that, in the case~$K=F$, we have seen that we could take~$f=1$). Now, we turn to the proof of (2) $\implies$ (4), establishing a little more.

\begin{prop} \label{prop-2-implies-4}
    With notation as in Proposition~\ref{prop-hilbert-symbol}, suppose there exists~$B_0$ with $\N_{F[\sq a]/F}(B_0)= b$. Then we may pick~$B$ such that~$\N_{F[\sq a]/F}(B) = b$, and such that the following holds. Let~$K= F[\sq a]$ and~$G = \Gal(K/F)$, and put 
\[ L= K[\sq{g(B)} : g \in G] \, .   \]
Then~$\Gal(L/F)$ can be identified with a subgroup of~$\f_p[\U_2] \rtimes \U_2$, yielding the homomorphism~$\gamma $ as in (4), that is, compatible with~$a$ and~$b$. Also $\gamma^{-1}(\f_p[\U_2]) = G_{K}$, and the image of the induced map 
\[ \h^1(\f_p[\U_2], \f_p) \longrightarrow \h^1(K, \f_p) \cong K^\times / K^{\times p}  \]
is the~$G$-module generated by the class of~$B$.

What is more, suppose we are in one of the following favourable cases: \begin{itemize}
\item[(i)] $a$ and $b$ are linearly independent in~$\lin F$;
\item[(ii)] $a$ is a~$p$-th power in~$F$;
  \item[(iii)] $p=2$.
\end{itemize}

Then we need not alter the given~$B_0$, that is, the above is true with~$B=B_0$.
\end{prop}

\begin{proof}
Some notation will be helpful. The group~$\U_2$ is cyclic of order~$p$, and we let~$\sigma $ be a generator. The elements of~$\f_p[\U_2]$ will be written using their coordinates in the basis~$1, \sigma \cdot 1, \ldots, \sigma^{p-1} \cdot 1$, so that a typical element of~$\f_p[\U_2] \rtimes \U_2$ is~$(x_1, \ldots, x_p) \sigma^i$ with~$x_k \in \f_p$ and~$i$ an integer. We have~$\sigma (x_1, \ldots, x_p) \sigma^{-1} = (x_p, x_1, x_2, \ldots, x_{p-1})$. The maps $s_1, s_2 \colon \f_p[\U_2] \rtimes \U_2 \to \f_p$ are given by 
\[ s_1((x_1, \ldots, x_p) \sigma^i) = i \, , \qquad s_2((x_1, \ldots, x_p) \sigma^i) = x_1 + \cdots + x_p \, .   \]
Also, a general comment is that the last statement of the Proposition (about the~$G$-module generated by~$B$) is obvious by construction, from Kummer theory (we merely wanted to highlight this fact). The arguments, in the various cases to be considered, will consist in choosing~$B$, describing~$\Gal(L/F)$, and embedding this group within~$\f_p[\U_2] \rtimes \U_2$. The homomorphism $\gamma $ will be the (pre-)composition of this embedding with the natural map~$G_F \to \Gal(L/F)$.

\medskip
  
We have essentially already dealt with (i): namely, Proposition~\ref{prop-from-Un-to-Un+1} for~$n=2$, with~$w= B= B_0$, gives us the result. Case (ii) is also easy. Here~$K=F$ and~$B=B_0= b$, so that~$L= F[\sq B]$. We have~$\Gal(L/F) \cong C_p$ (leaving aside the trivial case when~$b$ is also a~$p$-th power in~$F$), generated by an element~$\tau $ with~$\tau (\sq B) = \dz \sq B$, or in other words~$\chi_B(\tau )= 1$. We see~$\Gal(L/F)$ as a subgroup of~$\f_p[\U_2] \rtimes \U_2$ by mapping~$\tau $ to~$(1, 0, \ldots, 0)$. The compatibility can be verified easily. Finally, case (iii) is taken care of in~\cite[Prop.~2.3]{previous} .

\medskip

  We point out that, as far as Theorem~\ref{thm-main} below is concerned, we are done -- and after all, this is the main Theorem in the paper. The rest of the (rather long) proof is here to establish that (2) $\implies$ (4) in Proposition~\ref{prop-hilbert-symbol}, as promised, in absolutely all cases. We find the phenomenon of ``automatic realization'', which is taking place here, rather intriguing, so we provide all the details for completeness, even though this is quite a digression from Massey products.  

In the sequel, we are free to pick any~$B$ and ignore~$B_0$ -- its existence will not even be used. Indeed, outside of case (i) or (iii), the conditions (2) and (4) are both true, but for rather independent reasons.

\medskip

{\em The case when~$b$ is a~$p$-th power, but not~$a$}. Suppose~$b= B^p$ with~$B \in F$, so that~$b = \N_{K/F}(B)$. Then~$L= F[\sq a, \sq B]$. We have either~$\Gal(L/F) = C_p$ (consider then the subgroup of~$\U_3 ^{(1)}$ generated by~$\sigma $) or~$\Gal(L/F) = C_p \times C_p$ (use the subgroup generated by~$\sigma $ and~$(1, 1, \ldots, 1)$).

\medskip

{\em The case when~$a$ and~$b$ are colinear, but non-zero in~$\lin F$.} We will in fact write the proof in the case $a=b$, assuming that~$a$ is not a~$p$-th power in~$F$. The general case generates more notation, but is not fundamentally more complicated.  Recall that we assume that~$p$ is odd now. The first remark is that we may (and we do) take~$B= \sq a$: indeed 
\[ \N_{K/F}(\sq a) = \dz^{\frac{p(p-1)} {2}} \cdot a =  a = b \, .  \]
We have~$L= F[\sq{\dz}, \sqrt[p^2]{a}]$. We will let~$M = F[\sq{\dz}]$.
\begin{itemize}
\item {\em First subcase: $M= F$,} that is, we suppose first that~$F$ already had a primitive~$p^2$-th root of unity. Then~$L/F$ is cyclic, and its order is also the order of~$a$ in~$F^\times/ F^{\times p^2}$; if it were~$p$ (or~$1$), we would deduce that~$a$ is a~$p$-th power in~$F$, which it is not, by assumption. So~$\Gal(L/F) \cong C_{p^2}$. The element $(1, 0, \ldots, 0) \sigma  \in \f_p[\U_2] \rtimes \U_2$, where~$\sigma $ is a generator of~$\U_2 \cong C_p$, has order~$p^2$: indeed its~$p$-th power is~$(1, 1, \ldots, 1)$. The claim follows.

\item {\em Second subcase: } $a$ and $\dz$ are proportional in~$\lin F$, so in particular~$\sq{\dz} \not \in F$, and~$M =K \ne F$. In this case~$L= F[\sqrt[p^2]{\dz}] = F[\mu_{p^3}]$, and~$L/F$ is cyclotomic, with a cyclic Galois group. We remark that~$[L:F] = p^2$. Indeed, we have already pointed out that, for any~$x \in F^\times$, we have $\N_{F[\sq x]/F}(\sq x) = x$; as a result, if~$x$ is not a~$p$-th power in~$F$, then~$\sq x$ is not a~$p$-th power in~$F[\sq x]$, for we would get a contradiction upon taking norms down to~$F$. Applied to~$x= \dz$, we see that~$L \ne K$, and that~$[L:F]= p^2$ as claimed.

  The group~$\Gal(L/F)$ is generated by an element~$\tau $ with~$\tau (a^{1/p^2}) = \sq \dz \, a^{1/p^2}$, so~$\tau (\sq a) = \dz \sq a$, and thus~$\chi_a(\tau ) = 1$. We map~$\Gal(L/F)$ into $\f_p[\U_2] \rtimes \U_2$ by~$\tau \mapsto \tau ' = (1, 0,\ldots, 0) \sigma $. The element~$\tau '$ has order~$p^2$, and~$s_1(\tau ') = s_2(\tau') = 1$ as requested.

  \item {\em Third subcase: $\dz$ and~$a$ are linearly independent modulo $p$-th powers.} By Kummer theory, the element~$a$ is not a~$p$-th power in~$M$. It follows that the order of~$a$ in~$M^\times / M^{\times p^2}$ is not~$p$, and so~$\Gal(L/M)$ is cyclic of order~$p^2$. We have an exact sequence 
\[ 1 \longrightarrow \Gal(L/M)= C_{p^2} \longrightarrow \Gal(L/F) \longrightarrow \Gal(M/F) = C_p  \longrightarrow 1 \, .  \]
The extension~$F[\sqrt[p^2]{a}]/F$ is not normal, so~$\Gal(L/F)$ is not abelian. Just by checking which groups of order~$p^3$ exist (with~$p$ odd), we conclude that we must have $\Gal(L/F)= C_{p^2} \rtimes C_p$, and we note that there is only one such semidirect product, up to isomorphism.

Let~$\tau_1$ be an element of order~$p^2$, generating~$\Gal(L/M)$, and satisfying~$\tau_1(a^{1/p^2}) = \dz^{1/p} a^{1/p^2}$. In particular~$\tau_1(a^{1/p}) = \dz a^{1/p}$, and~$\chi_a(\tau_1) = 1$. We know that there exists an element~$\tau_2$  of order~$p$ in~$\Gal(L/F)$, restricting to a generator of~$\Gal(M/F)$. We can arrange to have~$\tau_2 \tau_1 \tau_2^{-1} = \tau_1^{p+1}$ (replacing~$\tau_2$ by a power of itself if necessary).

This relation, evaluated on~$a^{1/p^2}$, allows us to deduce that~$\tau_2(\dz^{1/p}) = \dz^{(p+1)/p}$, after a straightforward calculation. Then, using this, we let~$i$ be an integer such that~$\tau_2(a^{1/p^2}) = \dz^{i/p} a^{1/p^2}$, and compute that~$\tau_2^k(a^{1/p^2}) = \dz^{i m_k/p} a^{1/p^2}$ where 
\[ m_k = \frac{(p+1)^k - 1} {p} \, .   \]
In particular, the integer~$m_p$ is not divisible by~$p^2$. Since~$\tau_2$ must have order~$p$ nonetheless, we see that~$p$ divides~$i$. As a result, $\tau_2(a^{1/p}) = \dz^i a^{1/p} = a^{1/p}$. Thus~$\chi_a(\tau_2) = 0$.

Consider now the elements~$g= (1, 0, \ldots, 0) \sigma \in \f_p[\U_2] \rtimes \U_2$, and~$v= (0, 1, 2, \ldots, p-1) \in \f_p[\U_2]$. Then~$\sigma v \sigma^{-1}v^{-1} = (1, 1, \ldots, 1)$ and~$v^{-1} \sigma v = (2, 1, 1, \ldots, 1) \sigma = g^{p+1}$. It follows we have an embedding of~$\Gal(L/F)$ into $\f_p[\U_2] \rtimes \U_2$ with~$\tau_1 \mapsto g$, $\tau_2 \mapsto v^{-1}$. Here~$s_1(g) = s_2(g) = 1 = \chi_a(\tau_1)$, while~$s_1(v^{-1}) = s_2(v^{-1})= 0 = \chi_a(\tau_2)$, so the proof is complete.
\end{itemize}
\end{proof}

We have reached the main Theorem of this paper.

\begin{thm} \label{thm-main}
  Let~$F$ be a field with~$\dz \in F$, and let~$a, b, c, d \in F^\times$. Consider the following assertions.
  \begin{enumerate}
  \item There exists a Galois extension~$L/F$ with~$\Gal(L/F)$ identified with a subgroup of~$\tilde \U_5$, which is compatible with~$a, b, c, d$. In other words, there is a continuous homomorphism 
\[ \phi \colon G_F \longrightarrow  \tilde \U_5 \]
such that~$s_1 \circ \phi  = \chi_a$, $s_2 \circ \phi  = \chi_b$, $s_3 \circ \phi  = \chi_c$, $s_4 \circ \phi  = \chi_d$.

    \item One can find~$B \in F[\sq a]$ such that~$\N_{F[\sq a]/F}(B)= b f_1^p$ for some~$f_1 \in F^\times$, and~$C \in F[\sq d]$ such that~$\N_{F[\sq d]/F}(C)= c f_2^p$ for some~$f_2 \in F^\times$, with the property that for any~$\sigma \in \Gal(F[\sq a]/F)$ and any~$\tau \in \Gal(F[\sq d]/F)$, we have 
\[ (\sigma (B), \tau (C))_{F[\sq a, \sq d]} = 0 \, .   \]
  \end{enumerate}
  Then (1) $\implies$ (2). Moreover, we also have (2) $\implies$ (1) when~$p=2$, or when the following condition is satisfied: either~$a$ and $b$ are linearly independent in~$\lin F$,  or $a$ is a~$p$-th power, and likewise with~$d, c$ replacing~$a, b$ respectively. 
\end{thm}

\begin{proof}
  Suppose (1) holds, and let~$\gamma \colon G_F \to \tilde G$ denote the composition of~$\phi $ with the projection~$\tilde \U_5 \to \tilde G $. We write~$\gamma = \gamma_1 \times \gamma_2$, recalling that~$\tilde G = \U_3 ^{(1)} \times \U_3 ^{(2)}$.

  Apply Proposition~\ref{prop-4-implies-3} to~$\gamma_1$ first. This defines~$B \in F[\sq a]$, with appropriate norm, and~$B$ is in the image of 
\[ \gamma_1^* \colon \h^1(\f_p[\U_2], \f_p) \longrightarrow \lin {F[\sq a]} \cong \h^1(F[\sq a], \f_p) \, .   \]
It follows that the same can be said of all the~$\Gal(F[\sq a]/F)$-conjugates of~$B$. A similar discussion applies, when~$c, d, C, \gamma_2$ replace~$b, a, B, \gamma_1$ respectively.

Now apply Theorem~\ref{thm-maps-into-Un-tilde} to~$\gamma $. The existence of the lift $\phi $ is condition (1) from this Theorem, and we draw that the conjugates of~$B$ are orthogonal, for the cup-product, to all the conjugates of $C$. We have proved (2).

We turn to the converse, and so we assume that (2) holds, and that~$a, b, c, d$ satisfy the assumption of the Theorem (or that~$p=2$). Apply Proposition~\ref{prop-2-implies-4} to~$a$ and~$bf_1^p$, noting that we are in one of the three favourable cases: the existence of~$B$ shows that we can find a~$\gamma_1 \colon G_F \to  \U_3 ^{(1)}$ compatible with~$a$ and~$bf_1^p$, which of course is the same as saying that~$\gamma_1$ is compatible with~$a$ and~$b$. Moreover, the group~$\gamma_1^*(\h^1(\f_p[\U_2], \f_p))$ is the $\Gal(F[\sq a]/F)$-module spanned by~$B$, by the same Proposition. A similar discussion applies with~$a, b, B$ replaced by~$d, c, C$, so we have the existence of a~$\gamma_2 \colon G_F \to \U_3 ^{(2)}$ compatible with~$c, d$.

Now Theorem~\ref{thm-maps-into-Un-tilde} applied to~$\gamma = \gamma_1 \times \gamma_2$ shows the existence of a lift~$G_F \to \tilde U_5$ as in (1).
\end{proof}

We remark that, as~$B \in F[\sq a]$ and~$C \in F[\sq d]$ in this statement, condition (2) could be expressed as 
\[ (\sigma (B), \tau (C))_E = 0 \quad\textnormal{for all}~\sigma, \tau \in \Gal(E/F) \, ,   \]
where~$E= F[\sq a, \sq d]$. What is more, when~$a$ and~$d$ are linearly independent in~$\lin F$, then~$\Gal(E/F)$ splits off as a direct product, and the condition simplifies first to 
\[ (\rho  (B), \rho  (C))_E = 0 \quad\textnormal{for all}~ \rho  \in \Gal(E/F) \]
(take~$\rho = (\sigma, \tau )$ for~$\sigma \in \Gal(F[\sq a]/F)$ and $\tau \in \Gal(F[\sq d]/F))$. In turn, since the group~$\Gal(E/F)$ acts on~$\h^2(E, \f_p)$, this is really reduced to 
\[ (B, C)_E = 0 \, .   \]
Thus in good cases, the vanishing of a four-fold Massey product is controlled by the vanishing of a single cup-product.

\begin{coro}
Let~$a, b, c, d$ be as in the Theorem, and suppose they satisfy condition~(2). Then the Massey product 
\[ \langle a, a, \ldots, a, b, c, d, d, \ldots, d \rangle \, ,   \]
where~$a$ and~$d$ are repeated less than~$p$ times, vanishes.
\end{coro}

\begin{proof}
Theorem~\ref{thm-main} and Corollary~\ref{coro-repeated-massey}. 
\end{proof}

\section{Local fields} \label{sec-local-fields}

We proceed to describe a special situation when the (equivalent) conditions of Theorem~\ref{thm-main} hold, namely in the case of local fields with enough roots of unity. More precisely, we establish:

\begin{prop} \label{prop-local-fields}
Let~$p$ be a prime number, and let~$F$ be a local field which contains a primitive~$p$-th root of unity~$\dz$. Let~$a, b, c, d \in F^\times$ satisfy
\[ (a,b)_F= (b,c)_F = (c,d)_F = 0 \, .   \]
Finally, if~$[a], [b], [c]$ and $[d]$ span a~$1$-dimensional subspace in~$\lin F$, and if~$p>2$, then we assume further that $(a, \dz)_F = 0$ (this is automatic if~$F$ contains the~$p^2$-th roots of unity).

Then we can find~$B \in F[\sq a]$ such that~$\N_{F[\sq a]/F}(B)= b f_1^p$ for some~$f_1 \in F^\times$, and~$C \in F[\sq d]$ such that~$\N_{F[\sq d]/F}(C)= c f_2^p$ for some~$f_2 \in F$, with the property that for any~$\sigma \in \Gal(F[\sq a]/F)$ and any~$\tau \in \Gal(F[\sq d]/F)$, we have 
\[ (\sigma (B), \tau (C))_{F[\sq a, \sq d]} = 0 \, .   \]
In other words, condition (2) from Theorem~\ref{thm-main} is satisfied. 
\end{prop}

The rest of this section is dedicated to the proof. We shall also give a counter-example, showing that the hypothesis~$(a, \dz) = 0$ is necessary, in the case we indicate.

Let us recall a very useful observation: whenever~$E/F$ is an extension of local fields, with~$\dz \in F$, then the corestriction $\h^2(E, \f_p) \to \h^2(F, \f_p)$ is an isomorphism. Indeed, this follows from the more precise fact that, if we write~$\operatorname{Inv}_F$ for the Hasse isomorphism  of~$F$,
\[ \operatorname{Inv}_F \colon \h^2(F, \mathbb{G}_m) \longrightarrow \Q/\Z \, ,   \]
then~$\operatorname{Inv}_F(\cores(x)) = \operatorname{Inv}_E(x)$ for all~$x \in \h^2(E, \mathbb{G}_m)$. See for example~\cite[Chap.~XI, \S2, Prop.\ 1]{corps-locaux}.  (Recall that~$\h^2(F, \f_p)$ is the~$p$-torsion in~$\h^2(F, \mathbb{G}_m)$.) Together with Tate duality, these are the only facts we require of a ``local field''.

The proof of Proposition~\ref{prop-local-fields} itself is divided in cases. Common notation, however, includes~$E:= F[\sq a, \sq d]$, and~$\sigma $ and~$\tau $ always denote elements of~$\Gal(E/F)$. We speak of~$\lin F$ and~$\lin E$ additively, and write~$[f]_F$ for the class of~$f \in F^\times$ in~$\lin F$, or~$[f]$ when there is no risk of confusion; likewise for the notation~$[e]_E$.

The case~$p=2$ has already been dealt with elsewhere. Namely, from~\cite[Prop.~4.1]{localmasseycup}, we know that the conditions~$(a,b) = (b,c) = (c,d)=0$ are enough to imply the vanishing of~$\langle a, b, c, d \rangle$ whenever~$F$ is a local field (and for any~$p$). However, when~$p=2$, we have~$\tilde \U_5 = \U_5$, and so we have condition (1) of Theorem~\ref{thm-main}. By that Theorem, we have also (2). In the sequel, we assume that~$p>2$. 

\bigskip

{\em The non-degenerate case.} Let us assume for a start that~$[E:F] = p^2$. Pick any~$B, C$ such that~$\N_{F[\sq a]/F}(B) = b$, $\N_{F[\sq d]/F}(C)=c$, which is possible since~$(a,b) = (c,d) = 0$ (see Proposition~\ref{prop-hilbert-symbol}). We simply write that
\begin{align*}
  \cores_{E/F} (\sigma (B), \tau (C))_E & = \cores_{F[\sq a]/F} \circ \cores_{E/F[\sq a]}(\sigma (B), \tau (C))_E \\
  & = \cores_{F[\sq a]/F} (B, c)_{F[\sq a]} \\
  & = (b,c)_F = 0\, , 
\end{align*}
using the projection formula twice. Since the corestriction is injective, we have $(\sigma (B), \tau (C))_E = 0$, as we wanted.

\bigskip

{\em The case when one of~$[a]$ or~$[d]$ is~$0$}, say~$[a]_F= 0$. We take~$B= b$. If~$d$ is not a~$p$-th power, we see that~$(b, \tau (C))_E = 0$ since its corestriction is~$(b,c)_F= 0$. If, on the other hand, $d$ is also a~$p$-th power, we take~$C = c$ and there is nothing to check.

\bigskip

We continue the proof assuming that~$[a]_F$ and~$[d]_F$ are colinear but nonzero, so~$[E:F]= p$. We further subdivide into two cases.

\bigskip

{\em The case when~$[b]_E = [c]_E = 0$.} By Kummer theory, we see that~$[b]_F$ is in the span of~$[a]_F$ (or~$[d]_F$, of course), say~$[b]_F = i [a]_F$, and likewise~$[c]_F = j [a]_F$. In other words, we are in the case~$a, a^i, a^j, a^k$. Let~$\alpha = \sq a$ and~$B= \alpha ^i$, $C= \alpha ^j$. The norms of~$B$ and~$C$ are appropriate, and we have~$(B, C)_E = 0$ since~$B$ and~$C$ are colinear (and we assume that~$p$ is odd). If~$\sigma, \tau \in \Gal(E/F)$, then~$\sigma [B] = [B] + r[\dz]$, and~$\tau [C] = [C] + s [\dz]$, for some integers~$r, s$. To show~$(\sigma (B), \tau (C))_E = 0$, it suffices to establish that~$(\dz, C)_E = 0 = (B, \dz)_E$. In turn, the projection formula gives~$\cores_{E/F} (\dz, C) = (\dz, c)_F = j (\dz, a)_F = 0$ by assumption; since~$\cores$ is injective, we draw~$(\dz, C)_E = 0$, and~$(B, \dz) = 0$ is obtained in a similar fashion.

\bigskip

So we can finally turn to the most difficult case, when one of~$[b]_E$ or~$[c]_E$ is nonzero, say~$[c]_E \ne 0$. We keep assuming that~$[a]_F$ and~$[d]_F$ are colinear and nonzero. We pick some initial elements~$B, C \in E$ such that~$\N_{E/F}(B) = b $ and~$\N_{E/F} (C)= c$.

Some observations are in order. First, $\Gal(E/F) = C_p$, generated by~$\sigma $ (the letter loses its general meaning here), and we see~$M=\lin E$ as a~$C_p$-module. It is equipped with the cup-product, a non-degenerate bilinear form; note that it is~$C_p$-invariant in the sense $(\sigma (x), \sigma (y)) = (x, y)$ for~$x, y \in E^\times$ (indeed $C_p$ can only act trivially on $\h^2(E, \f_p) \cong \f_p$).

We call~$\rad(M)$ the image of~$\sigma - id$ (this is actually the radical in the sense of representation theory.) Note that the norm $\N_{E/F} \colon \lin E \to \lin F$ is zero on $\rad(M)$. A simple claim is then: {\em the orthogonal $\rad(M)^\perp$ is the module of fixed points $M^{C_p}$.}

{\em Proof:} If~$x$ is fixed by~$C_p$, then~$(x, \sigma (y) - y) = (x, \sigma (y)) - (x, y) = (\sigma^{-1}(x), y) - (x, y) = (x, y) - (x, y) = 0$. Conversely, if~$x$ is orthogonal to~$\rad(M)$, then~$(\sigma^{-1}(x)-x, y) = 0$ for all~$y$, by the same calculation. Since the form is non-degenerate, this implies~$\sigma^{-1}(x) = x$. The claim is proved.

Next, we observe that the condition~$[c]_E = \sum_i \sigma^i [C]_E \ne 0$ implies, by Lemma~\ref{lem-usual-lemma-rep-theory}, that the~$C_p$-module~$V$ generated by~$[C]$ is free.

Now, for~$1 \le i \le  p$ consider the linear form~$\ell_i$ on~$\rad(M)$ defined by~$\ell_i(x) = (x, \sigma^i(C))$. We claim that the span of the~$\ell_i$'s in~$\rad(M)^*$ has dimension~$p-1$, with relation~$\sum_i \ell_i = 0$. Indeed, if~$\sum \lambda_i \ell_i = 0$, then~$\sum \lambda_i \sigma^i[C]$ is orthogonal to~$\rad(M)$, and so by the previous claim, it is fixed by~$C_p$. As~$V$ is free, this happens if and only if ~$\lambda_1= \lambda_2= \cdots = \lambda_{p}$.

So, the forms~$\ell_1, \ldots , \ell_{p-1}$ being linearly independent, whatever the prescribed values~$\alpha_i \in \f_p$, we can find~$r \in \rad(M)$ with~$\ell_i(r) = \alpha_i$ for~$1 \le i < p$. Take~$\alpha_i = - (B, \sigma^i(C))$, so that 
\[ (rB, \sigma^i(C)) = \ell_i(r) + (B, \sigma^i(C)) = 0 \, .    \]
We replace~$B$ by~$B' = rB$, so that we still have~$\N_{E/F}[B']_E = [b]_F$, but now~$[B']$  is orthogonal to~$\sigma^i[C]$ when~$i$ is not~$0$ mod~$p$. However, the element~$[B']$ is also orthogonal to~$\sum_i \sigma^i [C]$, since 
\[ ( [B'], \sum_{i=0}^{p-1} \sigma^i [C] )_E = (B', c)_E \, ,   \]
while $\cores (B', c)_E= (b,c)_F = 0$, so~$(B', c)_E = 0$. We deduce~$(B', C) = 0$, so $(B', \sigma^i(C))=0$ with no restriction on~$i$. By the~$C_p$-invariance of the bilinear form, we do have~$(\sigma^i(B'), \sigma^j(C)) = 0$ for all~$i, j$, and we are done.

\bigskip

This concludes the proof of Proposition~\ref{prop-local-fields}.

\begin{rmk}
Here we show the necessity of having~$(a, \dz)=0$ when~$a, b, c, d$ are all colinear modulo~$p$-th powers.. Suppose that~$\dz \in F$ but that~$F$ does not contain a~$p^2$-th root of unity, where~$F$ is an~$\ell$-adic field for some~$\ell \ne p$. Select~$a \in F^\times$ so that~$\lin F$ is spanned by~$[a]$ and~$[\dz]$, and consider~$b=c=d=a$. Assuming~$p$ is odd, we certainly have~$(a,b) = (b,c) = (c,d)= 0$, but~$(a, \dz ) \ne 0$ by Tate duality.  We put~$E= F[\sq a]$ and~$\alpha = \sq a \in E$.

Now~$\cores(\alpha , \dz)_E = (a, \dz)_F \ne 0$, so we see that~$(\alpha , \dz) \ne 0$; say we have arranged to have~$(\alpha , \dz)= 1$ by choosing another root of unity, if necessary (this simplifies the calculations). Here we have used that~$\N_{E/F}(\alpha ) = a $ (this is for~$p$ odd).  Since~$\lin E$ is also of dimension~$2$, it must be spanned by~$[\alpha ]$ and~$[\dz]$, as these two classes are not orthogonal and thus not colinear.

We observe further that any~$B$ such that~$\N_{E/F}(B) = a$ up to~$p$-th powers is of the form~$[B] = [\alpha ] + \lambda  [\dz]$ with~$\lambda \in \f_p$. Likewise, if~$\N_{E/F}(C)= a$ up to~$p$-th powers, we have also~$[C] = [\alpha ] + \mu [\dz]$.

Now, there is~$\sigma \in \Gal(E/F)$ such that~$\sigma (\alpha ) = \alpha \dz$ , so~$\sigma [\alpha ] = [\alpha ] + [\dz]$ in additive notation. Thus~$\sigma^i[B] = [\alpha ] + (\lambda+i) [\dz]$ and~$\sigma^j[C] = [\alpha ] + (\mu + j) [\dz]$. We compute 
\[ (\sigma^i(B), \sigma^j(C)) = \lambda + \mu + i + j \, .   \]
Clearly, for some values of~$i$ and~$j$ this will be~$\ne 0$, whatever the initial choices of~$\lambda $ and~$\mu $. So~$B$ and~$C$, as in Proposition~\ref{prop-local-fields},  cannot be found.

\end{rmk}

\section{Splitting varieties}

We proceed to translate condition (2) in Theorem~\ref{thm-main} into polynomial equations. We continue to work with a field~$F$ with~$\dz \in F$, where~$p$ is a prime number, and we have distinguished elements~$a, b, c, d \in F^\times$.  It will be helpful to introduce 
\[ \E = F[X, Y]/(X^p - a, Y^p - d) \, , \]
as a substitute for~$E$, which in the previous section denoted~$F[\sq a, \sq d]$. (The script letter is a reminder that~$\E$ might not be a field.) Also, we put~$F_a = F[X]/(X^p - a)$ and view it as a subalgebra of~$\E$, and likewise for~$F_d = F[Y]/(Y^p - d)$. Finally, we pick~$p$-th roots~$\sq a$ and~$\sq d$ in some fixed algebraic closure of~$F$.

We will write~$\mu_p = \mu_p(F)$ for the group of~$p$-th roots of unity in~$F^\times$, generated by~$\dz$. 

\begin{prop} \label{prop-translation-splitting-variety}
  Let~$B \in F_a$ and~$C \in F_d$. The following conditions are equivalent.
  \begin{enumerate}
\item There exist elements $x_0, \ldots, x_{p-1} \in \E$ such that the following identity holds in the algebra~$\E[R]/(R^p - B)$ : 
\[ \prod_{\omega \in \mu_p} (x_0  + x_1 \omega R + x_2 \omega^2 R^2 + \cdots + x_{p-1} \omega^{p-1} R^{p-1}) = C \, .   \]

\item For any~$F$-homomorphism $\theta \colon \E \to K$, where~$K/F$ is an extension of fields, there exist elements $x_0, \ldots, x_{p-1} \in K$ such that the following identity holds in~$K[R]/(R^p - \theta (B))$ : 
\[ \prod_{\omega \in \mu_p} (x_0  + x_1 \omega R + x_2 \omega^2 R^2 + \cdots + x_{p-1} \omega^{p-1} R^{p-1}) = \theta (C) \, .   \]

\item For any~$F$-homomorphism $\theta \colon \E \to K$, where~$K/F$ is an extension of fields, we have 
\[ (\theta (B), \theta (C))_K = 0 \, .   \]

\item Put~$E= F[\sq a, \sq d]$. Then for any~$\sigma, \tau \in \Gal(E/F)$, we have 
\[ (\sigma (B), \tau  (C))_E = 0 \, ,   \]
where we view~$B$ and~$C$ as elements of~$E$ by substituting~$\sq a$ for~$X$, and~$\sq d$ for~$Y$.
\end{enumerate}
\end{prop}

%% The reader should keep in mind that we only care about the equivalence of (1) and (4), but the other, intermediate conditions seem to shed some light on the situation.

\begin{proof}
  (1) $\implies$ (2) is trivial: simply apply~$\theta $. The converse follows from the fact that the étale algebra~$\E$ is isomorphic to a direct sum of fields such as~$K$.

  We turn to the equivalence of (2) and (3). In fact, let~$\theta \colon \E \to K$ be as proposed, and we prove that the two conditions, with this fixed~$\theta $, are equivalent. If~$\theta (B)$ is not a~$p$-th power in~$K$, we use that the identity in (2) means precisely that~$\theta (C)$ is a norm from~$K\left[\sq{\theta (B)}\right]$. This is equivalent to $(\theta (B), \theta (C))_K =0$ by Proposition~\ref{prop-hilbert-symbol}.

  If~$\theta (B)$ is a~$p$-th power in~$K$, things are slightly different. The condition expressed in (3) is essentially vacuous. We must show that the equation in (2) has a solution. For this, we consider the matrix equation 
\[ \left(\begin{array}{ccccc}
  1 & \omega_1 \sq{B} & \omega_1^2 \sq{B}^2 & \cdots & \omega_{1}^{p-1} \sq{B}^{p-1} \\
1 & \omega_2 \sq{B} & \omega_2^2 \sq{B}^2 & \cdots & \omega_{2}^{p-1} \sq{B}^{p-1} \\
\vdots & \vdots &     &   &    \\
1 & \omega_p \sq{B} & \omega_p^2 \sq{B}^2 & \cdots & \omega_{p}^{p-1} \sq{B}^{p-1} 
\end{array}\right)  \left(\begin{array}{c}
x_0 \\ x_1 \\ \vdots \\ x_{p-1}
\end{array}\right) = \left(\begin{array}{c}
\theta (C) \\ 1 \\ \vdots \\ 1
\end{array}\right) \]
Here we have chosen an enumeration~$\omega_1, \ldots, \omega_p$ of~$\mu_p$. Note that all three matrices have coefficients in~$K$. Since the square matrix has a Vandermonde determinant which is nonzero, as~$\omega_i \sq{B} \ne \omega _j \sq{B}$ when~$i \ne j$, we conclude that we may find a solution~$(x_0, \ldots, x_{p-1}) \in K^p$ to the system.

So the equation proposed in (2) has a solution, with~$R$ replaced by~$\sq B$. However, this is visibly independent of the choice of~$p$-th root~$\sq B$ (in fact, the equation is always the same!). As~$K[R]/(R^p - B)$ is now a direct sum of~$p$ copies of~$K$, corresponding to these different choices, we see that there is a solution in~$K[R]/(R^p - B)$ as well.

The equivalence of (3) and (4) is trivial.  
\end{proof}

\begin{rmk}
Note that the algebra~$\E[R]/(R^p - B)$ showing up in (1) is a free~$\E$-module on~$1, R, \ldots, R^{p-1}$, so that the proposed equation amounts to~$p$ different identities in~$\E$. For example, suppose~$p=2$. Then the equation in (1) is 
\[ (x_0 + x_1R)(x_0 - x_1 R) = x_0^2 - x_1^2 R^2 = x_0^2 - B x_1^2  = C\, ,   \]
since~$R^2 = B$. We do have two equations, but comparing the coefficients of~$R$ merely gives~$0=0$.
\end{rmk}

We can now write down a splitting variety for our problem. We have variables $\beta_0$, $\ldots$, $\beta_{p-1}$, $\gamma_0$, $\ldots $, $\gamma_{p-1}$, $f_1$, $f_2$ in~$F$, as well as variables~$x_0$, $\ldots $, $x_{p-1}$ in~$\E$, and the variety is defined over~$F$ by the equations $f_1 \ne 0$, $f_2 \ne 0$, and:

\begin{enumerate}
\item[(V1)] ``The norm of~$B = \beta_0 + \beta_1 X + \cdots + \beta_{p-1}X^{p-1}$ is~$bf_1^p$'', so in~$F_a$:
\[ \prod_{\omega \in \mu_p} (\beta_0 + \beta_1 \omega X + \cdots + \beta_{p-1} \omega ^{p-1} X^{p-1}) = b f_1^p \, .   \]

\item[(V2)] ``The norm of~$C = \gamma_0 + \gamma_1 Y + \cdots + \gamma_{p-1}Y^{p-1}$ is~$cf_2^p$'', so in~$F_d$:
\[ \prod_{\omega \in \mu_p} (\gamma_0 + \gamma_1 \omega Y + \cdots + \gamma_{p-1} \omega ^{p-1} Y^{p-1}) = c f_2^p \, .   \]

\item[(V3)] ``$(\sigma (B), \tau (C)) = 0$ for all~$\sigma, \tau $'', so by the Proposition, in~$\E[R]/(R^p - B)$ : 
\[ \prod_{\omega \in \mu_p} (x_0  + x_1 \omega R + x_2 \omega^2 R^2 + \cdots + x_{p-1} \omega^{p-1} R^{p-1}) = C \, .   \]
\end{enumerate}

\begin{thm}
  Let~$F$ be a field with~$\dz \in F$, and let~$a, b, c, d \in F^\times$. Let~$V$ be the affine algebraic variety defined by the equations (V1)-(V2)-(V3). The following assertions are equivalent. \begin{enumerate}
  %% \item There exists a Galois extension~$L/F$ with~$\Gal(L/F)$ identified with a subgroup of~$\tilde \U_5$, which is compatible with~$a, b, c, d$ (as in \S\ref{sec-prelim}). 

  \item The variety~$V$ has an~$F$-rational point, or in other words, there is a solution to (V1)-(V2)-(V3) over~$F$.

    \item Condition (2) from Theorem~\ref{thm-main}. 
  \end{enumerate}

\end{thm}

Of course, from Theorem~\ref{thm-main}, for generic values of~$a, b, c, d$, the conditions expressed here are also equivalent to the vanishing of the Massey product of~$a, b, c, d$ in the sense of~$\tilde \U_5$. 

\begin{proof}
  Suppose first that neither~$a$ nor~$d$ is a~$p$-th power in~$F$. Then we can identify~$F_a$ and $F_d$  with~$F[\sq a]$ and $F[\sq d]$ respectively, after choosing once and for all favourite~$p$-th roots~$\sq a, \sq d$ in an algebraic closure of~$F$. The argument is then very easy. Namely, equation (V1) controls the existence of~$B \in F[\sq a]$ such that~$\N_{F[\sq a]/F}(B) = b f_1^p$, and similarly equation (V2) controls the existence of~$C \in F[\sq d]$ such that~$\N_{F[\sq d]/F}(C) = c f_2^p$, while equation (V3) is equivalent to the condition $(\sigma (B), \tau  (C))_{F[\sq a, \sq d]} = 0$ for all~$\sigma , \tau $, by Proposition~\ref{prop-translation-splitting-variety}.

  %% Thus Theorem~\ref{thm-main} says precisely that the present conditions (1) and (2) are equivalent.

Now let us give an argument for the case when~$a$ is a~$p$-th power in~$F$, but~$d$ is not (the other cases are treated similarly). Suppose (2) holds. This gives us the element~$C$ showing up in equation (V2). We have not enough about equation (V1) just yet, but at least we draw that~$(b, \tau (C))_{F[\sq d]} = 0$ for all~$\tau \in \Gal(F[\sq d]/F)$. Now, let us find a solution to (V1) using the Vandermonde argument used in the previous proof. We state this as follows: let the homomorphisms $F_a \to F$ be written~$\theta_\omega $, indexed by the~$\omega \in \mu_p$ in the natural way (that is $\theta_\omega (X) = \omega \sq a$). Then the Vandermonde argument shows the existence of~$B \in F_a$ such that~$\theta_1(B) = b$ and~$\theta_\omega (B)= 1$ for~$\omega \ne 1$. This~$B$ satisfies (V1) with~$f_1= 1$, since applying~$\theta_\omega $ to both sides of (V1) always gives~$b$, whatever~$\omega $. Further, by inspection, for any~$\theta \colon \E \to K$ we have 
\[ (\theta (B), \theta (C))_{K} = 0 \, ,   \]
as in condition (3) of Proposition~\ref{prop-translation-splitting-variety}. By this Proposition, there is a solution to (V3). We have proved (2) $\implies$ (1).

Now suppose (1). Again (V2) controls the norm of~$C$. From (V1) we draw the existence of some~$B_0$ with
\[ \prod_\omega \, \theta_\omega (B_0) = bf_1^p \, .   \]
Since (V3) has a solution, Proposition~\ref{prop-translation-splitting-variety} guarantees that 
\[ (\theta_\omega (B_0), \tau (C))_{F[\sq d]} = 0  \]
for all~$\omega \in \mu_p$ and all~$\tau \in \Gal(F[\sq d]/F)$. Taking the sum of these as~$\omega $ varies, we draw $(b, \tau (C)) = 0$. We have obtained condition (2) from Theorem~\ref{thm-main}, with~$B=b$.
\end{proof}

\begin{coro}
Let~$F$ be a number field containing the~$p^2$-th roots of unity, and let~$a, b, c, d \in F^\times$ satisfy 
\[ (a,b)_F = (b, c)_F = (c, d)_F = 0 \, .   \]
Finally, let~$V$ be the variety constructed above. Then for each place~$v$ of~$F$, we have~$V(F_v)\ne \emptyset$.
\end{coro}

\begin{proof}
Apply the Theorem, and Proposition~\ref{prop-local-fields}. 
\end{proof}

Thus, if~$V$ satisfies a local-to-global principle, then we deduce that the Massey product of~$a, b, c, d$ vanishes in the sense of~$\tilde \U_5$, at least for generic~$a, b, c, d$, by Theorem~\ref{thm-main}. 

%% \vfill
%% \pagebreak

\bibliography{myrefs}

\newcommand{\noopsort}[1]{} \newcommand{\printfirst}[2]{#1}
  \newcommand{\singleletter}[1]{#1} \newcommand{\switchargs}[2]{#2#1}
  \def\cprime{$'$}
\begin{thebibliography}{MT17b}

\bibitem[Dwy75]{dwyer}
W.~G. Dwyer.
\newblock Homology, {M}assey products and maps between groups.
\newblock {\em J. Pure Appl. Algebra}, 6(2):177--190, 1975.

\bibitem[EM]{EM-triple}
I.~Efrat and E.~Matzri.
\newblock Triple {M}assey products and absolute {G}alois groups.
\newblock {\em J. Eur. Math. Soc.}
\newblock to appear, arXiv:1412.7265.

\bibitem[GMT]{previous}
P.~Guillot, J.~Min{\'a}{\v{c}}, and A.~Topaz.
\newblock Four-fold {M}assey products in {G}alois cohomology.
\newblock with an appendix by Olivier Wittenberg, to appear.

\bibitem[HW15]{hopkins}
M.~J. Hopkins and K.~G. Wickelgren.
\newblock Splitting varieties for triple {M}assey products.
\newblock {\em J. Pure Appl. Algebra}, 219(5):1304--1319, 2015.

\bibitem[MSS08]{automatic}
J.~Min\'a\v{c}, A.~Schultz, and J.~Swallow.
\newblock Automatic realizations of {G}alois groups with cyclic quotient of
  order {$p^n$}.
\newblock {\em J. Th\'eor. Nombres Bordeaux}, 20(2):419--430, 2008.

\bibitem[MT16]{MT16}
J.~Min\'a\v{c} and N.~D. T\^an.
\newblock Triple {M}assey products vanish over all fields.
\newblock {\em J. Lond. Math. Soc. (2)}, 94(3):909--932, 2016.

\bibitem[MT17a]{MT3}
J.~Min{\'a}{\v{c}} and N.~D. T{\^a}n.
\newblock Construction of unipotent {G}alois extensions and {M}assey products.
\newblock {\em Adv. Math.}, (304):1021--1054, 2017.

\bibitem[MT17b]{localmasseycup}
J.~Min\'a\v{c} and N.~D. T\^an.
\newblock Counting {G}alois {$\Bbb{U}_4(\Bbb{F}_p)$}-extensions using {M}assey
  products.
\newblock {\em J. Number Theory}, 176:76--112, 2017.

\bibitem[MT17c]{MT17}
J.~Min\'a\v{c} and N.~D. T\^an.
\newblock Triple {M}assey products and {G}alois theory.
\newblock {\em J. Eur. Math. Soc. (JEMS)}, 19(1):255--284, 2017.

\bibitem[Ser79]{corps-locaux}
J.-P. Serre.
\newblock {\em Local fields}, volume~67 of {\em Graduate Texts in Mathematics}.
\newblock Springer-Verlag, New York-Berlin, 1979.

\bibitem[Wha57]{whaples}
G.~Whaples.
\newblock Algebraic extensions of arbitrary fields.
\newblock {\em Duke Math. J.}, 24:201--204, 1957.

\end{thebibliography}
\bibliographystyle{custom}

\end{document}